\documentclass[11pt]{article}
\usepackage{amsfonts,latexsym,rawfonts,amsmath,amssymb,amsthm, a4wide, times, color}
\usepackage{graphicx}
\usepackage[plainpages=false]{hyperref}
\numberwithin{equation}{section}
\usepackage[titletoc,title]{appendix}

\numberwithin{equation}{section}

\newcommand{\pd}[2]{\frac {\partial #1}{\partial #2}}

\newcommand{\al}{\alpha}
\newcommand{\bb}{\beta}
\newcommand{\la}{\lambda}
\newcommand{\La}{\Lambda}
\newcommand{\oo}{\omega}

\newcommand{\dd}{\delta}
\newcommand{\Na}{\nabla}

\def\ga{\gamma}

\newcommand{\ee}{\epsilon}

\newcommand{\Si}{\Sigma}

\newcommand{\te}{\theta}

\newcommand{\beq}{\begin{equation}}
    \newcommand{\eeq}{\end{equation}}
\newcommand{\beqs}{\begin{eqnarray*}}
    \newcommand{\eeqs}{\end{eqnarray*}}
\newcommand{\beqn}{\begin{eqnarray}}
    \newcommand{\eeqn}{\end{eqnarray}}
\newcommand{\beqa}{\begin{array}}
    \newcommand{\eeqa}{\end{array}}

\def\td{\tilde}

\def\RR{{\mathbb R}}

\def\ri{\rightarrow}

\def\un{\underline}

\def\no{{\nonumber}}

\def\pbp{\sqrt{-1}\partial\bar\partial}
\def\tr{{\rm tr}}

\def\vol{{\rm vol}}

\def\cH{{\mathcal H}}

\def\cK{{\mathcal K}}

\def\ka{{\kappa}}

\def\Rc{{Ric(\oo_g)}}

\newtheorem{prop}{Proposition}[section]
\newtheorem{theo}[prop]{Theorem}
\newtheorem{lem}[prop]{Lemma}

\newcommand{\RomanNumeral}[1]{\uppercase\expandafter{\romannumeral#1}}
\makeatletter
\def\ExtendSymbol#1#2#3#4#5{\ext@arrow 0099{\arrowfill@#1#2#3}{#4}{#5}}

\makeatother

\makeatletter
\def\ExtendSymbol#1#2#3#4#5{\ext@arrow 0099{\arrowfill@#1#2#3}{#4}{#5}}

\makeatother

\title{Calabi flow with bounded $L^p$ scalar curvature (II)
}
\author{ Haozhao Li \footnote{Supported by NSFC grant No. 12471058,  the CAS Project for Young Scientists
        in Basic Research (YSBR-001), and the Fundamental Research Funds
        for the Central Universities.} \quad  and \quad Linwei Zhang }

\begin{document}
    \bibliographystyle{plain}

    \date{}

    \maketitle

\begin{abstract}
In this paper, we show that on a compact K\"ahler manifold   the Calabi flow can be extended as long as some space-time $L^p$ integrals of the scalar curvature  are bounded.
    \end{abstract}

    \tableofcontents

\section{Introduction}

This paper is the continuation of the study on the extension of Calabi flow in \cite{[LZZ]}. In \cite{[LZZ]}, based on Chen-Cheng's estimates in \cite{[CC3]}, we showed that the Calabi flow can be extended as long as the $L^p$ norm of the scalar curvature is bounded. The estimates in \cite{[LZZ]} are essentially elliptic. In this paper, we want to use the parabolic structure of the Calabi flow equation to study the extension of Calabi flow under some space-time integrals of the scalar curvature as in other second order geometric flows, such as Ricci flow and mean curvature flow etc.

Let $(M^n, g)$ be a compact K\"ahler manifold of complex dimension $n$. To study the constant scalar curvature metrics in a K\"ahler class, E. Calabi in \cite{[Cal1]} introduced the Calabi flow, which is the gradient flow of the Calabi energy. We call a family of K\"ahler metrics $\oo_{\varphi(t)}(t\in [0, T])$ in the same K\"ahler class $[\oo_g]$   a solution of
Calabi flow, if the K\"ahler potential $\varphi(t)$ satisfies the equation
\beq
\pd {\varphi(t)}t=R(\oo_{\varphi(t)})-\un R, \label{eq000}
\eeq where $R(\oo_{\varphi(t)})$ denotes the scalar curvature of the metric $\oo_{\varphi(t)}$ and $\un R$ denotes the average of the scalar curvature.
The Calabi flow is expected to be
an
effective tool to find constant scalar curvature metrics in a K\"ahler class. However,  since the Calabi flow   a fully nonlinear fourth order partial differential equation, it is difficult to study its behavior by  standard parabolic estimates. In this paper, we continue to study the extension problem of Calabi flow under some conditions on the scalar curvature.

There are many literatures on Calabi flow. The long time existence and convergence of Calabi flow on Riemann surfaces is completely solved by Chrusciel \cite{[Chru]},    Chen
\cite{[Chen]} and   Struwe \cite{[Stru]} independently by  using different methods. In  \cite{[ChenHe1]},  Chen-He  showed  the short time existence and stability results of Calabi flow in general K\"ahler manifolds of higher dimensions. In a series of papers \cite{[ChenHe3]}\cite{[ChenHe2]}\cite{[He4]}\cite{[He5]}, Chen and He studied the long time existence
and convergence under some curvature conditions. Moreover,  Tosatti-Weinkove  \cite{[TW]} proved the long time existence and convergence under the assumption that  the Calabi energy is small. Szekelyhidi in \cite{[Sz2]} studied the Calabi flow on ruled surfaces, and in \cite{[Sz]} studied the Calabi flow under the assumption
that the curvature tensor is uniformly bounded and the $K$-energy is proper.   Streets
\cite{[St1]}\cite{[St2]}  showed the long time existence of a weak solution to the
Calabi flow and Berman-Darvas-Lu \cite{[BDL]} showed the convergence of weak Calabi flow on general K\"ahler manifolds.

A conjecture of X. X. Chen in \cite{[Chen2]} says that the Calabi flow always exists for all time for any initial K\"ahler metrics.
Chen-He's
result in \cite{[ChenHe1]} showed the extension result of Calabi flow under the assumption that the Ricci curvature stays bounded, and  Huang in \cite{[Huang]} proved the extension results of the Calabi
flow on toric manifolds.
 In \cite{[LZ]} Li-Zheng showed the long time existence under the assumptions on the lower boundedness of Ricci curvature, the properness of the $K$-energy, and the $L^p(p>n)$ bound of scalar curvature.
 In \cite{[LWZ]}, Li-Wang-Zheng used the  ideas from Ricci flow in
\cite{[Wang3]} and \cite{[Wang2]} to study the convergence of Calabi flow. A breakthrough was made by Chen-Cheng in \cite{[CC3]} and they showed that the Calabi flow always exists as long as the scalar curvature is bounded.\\

In the previous paper \cite{[LZZ]}, Li-Zhang-Zheng proved that the Calabi flow can be extended as long as the $L^p$ scalar curvature is bounded.   In this paper, based on Chen-Cheng's estimates in \cite{[CC3]} we show that Calabi flow can be extended as long as some space-time $L^p$ integrals of the scalar curvature  are bounded. The  main theorem in this paper is the following result.

\begin{theo}\label{theo:1.1} Let $(M, \oo_g)$ be a compact K\"ahler manifold of complex dimension $n\geq 2$, and
  $\{\varphi(t), t\in [0, T)\}$ the solution to the Calabi flow (\ref{eq000}) with $T<\infty.$
 If  the scalar curvature satisfies
\beq
\int_0^T\int_M\;\Big((\Delta_{\varphi}R)^{p+1}+|R|^{2p}\Big)\,\oo_{\varphi}^ndt\leq C, \label{eq:theo1}
\eeq   for $p>n$,  the Calabi flow  can be extended past time $T$.

 \end{theo}

In Theorem \ref{theo:1.1}, we need to assume a technical condition on the space-time $L^p$ bound of $\Delta_\varphi R$, which seems inevitable if we calculate the time derivative of the evolving metrics. It is possible that the condition on $\Delta_\varphi R$ can be replaced by some other geometric conditions, and we will discuss this problem in future papers.

Theorem \ref{theo:1.1} is similar to the results in other geometric flows such as Ricci flow and mean curvature flow. For Ricci flow, B. Wang \cite{[Wang2]} proved that on a compact Riemannian manifold of real dimension $m$ the Ricci flow can be extended if
\beq
\int_0^T\int_M\;|Rm|^{p}\,\oo_{\varphi}^ndt\leq C, \quad p\geq \frac {m+2}{2}. \no
\eeq
G. Di Matteo \cite{[Matt]} extends Wang's result to some mixed integral norms of the curvature tensor. For mean curvature flow, Xu-Ye-Zhao \cite{[XYZ]} proved that the mean curvature flow $\Si^m_t\subset \RR^{m+1}$ can be extended if
\beq
\int_0^T\int_M\;|A|^{p}\,d\mu dt\leq C, \quad p\geq m+2. \no
\eeq
Le-Seum \cite{[LS0]} also showed some extension results of mean curvature flow under some mixed
integral norms of the second fundamental form. Since Ricci flow and mean curvature flow are second-order geometric flows, the usual parabolic Moser iteration argument applies once the Sobolev inequality holds. However, since Calabi flow is a fourth-order flow, we need to overcome  new difficulties.\\

We outline the proof of Theorem \ref{theo:1.1}. The proof is divided into several steps:
\begin{enumerate}
  \item[(1).] The $C^0$ estimates of $F$ and $\varphi$.  Lu-Seyyedali \cite{[LS]} proved the $C^0$ estimates of $F$ and $\varphi$ under the assumption that the $L^p(p>n)$ norm of the scalar curvature is bounded. In the proof of Theorem \ref{theo:1.1} we use the parabolic version of Lu-Seyyedali's argument to show that $\|F\|_{C^0}$ and $\|\varphi\|_{C^0}$ are bounded under the condition (\ref{eq:theo1}) of Theorem \ref{theo:1.1}. Recall that using the method of Guo-Phong-Tong \cite{[GPT]}, Chen-Cheng in \cite{[CC1]} proved the $L^{\infty}$ estimate of the parabolic complex Monge-Ampere flow:
\beq
-\pd {\psi}t(\oo_g+\pbp \psi)^n=e^G\oo_{g}^n.
\eeq
Based on Chen-Cheng's estimates, we show that $\|F\|_{L^{1+\dd}(M\times [0, T), \oo_{\varphi})}$ is uniformly bounded along the flow. This together with the assumption of Theorem \ref{theo:1.1} implies that $\|\varphi\|_{C^0}$ is bounded along the Calabi flow. Thus using the parabolic maximum principles we show that $\|F\|_{C^0}$ is bounded.

\item[(2). ] Higher order estimates of $F$ and $\varphi.$ We follow the argument of Chen-Cheng \cite{[CC3]}, Li-Zhang-Zheng \cite{[LZZ]} and the parabolic Moser iteration to show that the space-time quantities
\beq
\int_0^T\int_M\;(n+\Delta_g\varphi)^q\,\oo_g^ndt,\quad  \int_0^T\int_M\;|\Na F|_{\varphi}^{2\ka}\,\oo_{\varphi}^ndt
\eeq are bounded for some $\ka>2n$ and any $q\geq 1$. Using these estimates and the
parabolic Moser iteration argument, we show that $\|n+\Delta_g\varphi\|_{C^0}$ is bounded. Thus, using similar argument as in Chen-Cheng \cite{[CC3]} the higher order estimates of $F$ and $\varphi$ can be obtained. The argument is based on the Sobolev inequality of Guo-Phong-Song-Sturm \cite{[GPSS]} or Guedj-T\^o \cite{[GT1]}.
\end{enumerate}

 The organization of this paper is as follows. In Sec. 2 we recall some basic notations and show the parabolic Sobolev inequality  on K\"ahler manifolds. In Sec. 3 we first show the $L^{\infty}$ norm of $F$ and $\varphi$, and then we show the  space-time $L^p$ estimates of $n+\Delta_{g}\varphi$ and $|\Na F|_{\varphi}$, which implies the $L^{\infty}$ norm of $n+\Delta_{g}\varphi$. Finally, in Sec. 4 we show the higher-order estimates along the Calabi flow.

\section{Preliminary results}
In this section, we recall some basic notations and results  on K\"ahler manifolds.
Let $(M, \oo_g)$ be a compact K\"ahler manifold with complex dimension $n$. We define the space of K\"ahler potentials
\beq
\cH(\oo_g)=\{\varphi\in C^{\infty}(M, \RR)\;|\;\oo_g+\pbp\varphi>0\},
\eeq
and we define the subset $\cH_0$ of $\cH(\oo_g)$ by
\beqn
\cH_0:=\{\varphi\in\cH(\oo_g)\;|\; I_{\oo_g}(\varphi)=0\},
\eeqn
where the functional $I_{\oo_g}$ is defined by
\beqs
I_{\oo_g}(\varphi)=\frac{1}{(n+1)!}\sum_{k=0}^{n}\int_M\varphi \oo^k\wedge\oo_\varphi^{n-k}.
\eeqs
It is clear that for any path $\varphi(t)\in \cH$, we have
\beq
\frac d{dt}I_{\oo_g}(\varphi(t))=\frac 1{n!}\int_M\; \pd {\varphi(t)}t\oo_{\varphi(t)}^n. \label{eq:I}
\eeq
The $K$-energy is defined by
\beq
\cK(\varphi)=-\int_0^1\,\int_M\;\pd {\varphi_t}t(R(\oo_{\varphi_t})-\un R)\,\frac {\oo_{\varphi_t}^n}{n!}.
\eeq Note that along the Calabi flow we have
\beq
\frac d{dt}\cK(\varphi(t))=-\int_M\;(R(\oo_{\varphi(t)})-\un R)^2\,\frac {\oo_{\varphi_t}^n}{n!}\leq 0.
\eeq Therefore, the $K$-energy is non-increasing along the Calabi flow.
It is known that the $K$-energy can be written as
\beq
\cK(\varphi)=\int_M\; \log \frac {\oo_{\varphi}^n}{\oo_g^n}\,\frac {\oo_{\varphi}^n}{n!}+J_{-Ric(\oo_g)}(\varphi),
\eeq where for a $(1, 1)$ form $\chi$, we define
 $$J_{\chi}(\varphi)=\int_0^1\,\int_M\;\pd {\varphi_t}t\Big(\chi\wedge \frac {\oo_{\varphi_t}^{n-1}}{(n-1)!}-\un \chi \frac {\oo_{\varphi_t}^n}{n!}\Big)\oo_{\varphi_t}^n\wedge dt,$$ where $\varphi_t\in \cH$ is a path connecting $0$ and $\varphi.$ Here
\beq
\un \chi=\frac {\int_M\;\chi\wedge \frac {\oo_g^{n-1}}{(n-1)!}}{\int_M\;\frac {\oo_g^n}{n!}}.
\eeq
For any function $\varphi\in \cH(\oo_g)$, we define the function $F$ by
\beq
(\oo_g+\pbp\varphi)^n =e^F\oo_g^n.
\eeq
Let $\varphi(x, t)$ be a family of K\"ahler potentials.  We denote by $R$ the scalar curvature of the metric $\oo_{\varphi(x, t)}$, and $R_g$ to denote the scalar curvature of the metric $\oo_g$. For simplicity, we write
\beqn
\|f\|_s&=&\Big(\int_0^T\,\int_M\;|f(x, t)|^s\,\oo_{\varphi(x, t)}^n\,dt\Big)^\frac 1s,\no\\
\|f\|_{s, t}&=&\Big(\int_M\;|f(x, t)|^{s}\,\oo_{\varphi(x, t)}^n\Big)^{\frac{1}{s}}.\no
\eeqn
We denote by $|\Na f|_{\varphi}$ (resp. $|\Na f|_g$) the norm of the gradient of $f$ with respect to the metric $\oo_{\varphi}$ (resp. $\oo_g$). Moreover, we denote by $\Delta_{\varphi}$ (resp. $\Delta_g$) the Laplace operator with respect to the metric $\oo_{\varphi}$ (resp. $\oo_g$).

Now we recall the following interpolation inequality.

\begin{lem}\label{lem:2.1}(cf. \cite[Equations (7.9) and (7.10)]{GT]}, \cite[Lemma 2.1]{[LZZ]}) If $0<p<r<q$,   for any $\ee>0$ we have
\beq
\|f\|_{r, t}\leq\|f\|_{q, t}^{\te}\|f\|_{p, t}^{1-\te},
\eeq where $\te=\frac {(r-p)q}{(q-p)r}\in (0, 1)$.

\end{lem}

        Following Guo-Phong-Song-Sturm \cite{[GPSS]} or Guedj-T\^o \cite{[GT1]}, the Sobolev
constant of the metric $\oo_{\varphi}$ is bounded under some conditions.

    \begin{theo}(cf. \cite[Theorem 2.1]{[GPSS]},  \cite[Theorem 2.6)]{[GT1]} \label{theo:2.1}
            For any $\gamma\in(1,\frac {n}{n-1})$ and $u\in W^{1, 2}(M, \oo_{\varphi})$, we have the Sobolev inequality with respect to the metric $\oo_\varphi$
            \beqn
            \Big(\int_M\,|u|^{2\gamma}\;\oo_\varphi^n\Big)^{\frac 1{\gamma}}\leq C(n,\oo_g,\gamma,\|F\|_\infty)\int_M \,(|u|^2+|\Na u|_{\varphi}^2)\;\oo_\varphi^n.
            \eeqn
        \end{theo}

It is known that the following parabolic Sobolev inequality follows from Theorem \ref{theo:2.1}, and we collect the proof for the readers' convenience.

    \begin{lem}\label{lem:interF}
        For any $0<\kappa<2<\beta<\gamma<\frac{2n}{n-1}$ and $u\in W^{1, 2}(M\times [0, T), \oo_{\varphi})$, we have
        \beq
        \int_{0}^{T}\;dt\int_M|u|^\beta\;\oo_{\varphi}^n\leq C\sup_{t\in [0,T)}\|u\|_{\kappa,t}^{(1-\frac 2\gamma)\kappa}\int_{0}^{T}\;dt\int_M\;\Big(|\nabla u|_\varphi^2+|u|^2\Big)\;\oo_{\varphi}^n.\label{eq:0.1}
        \eeq
        where $C$ depends on $\oo_g, n,\|F\|_\infty$ and $\gamma$. Moreover, the constants  $\theta\in(0,1)$, $\kappa, \beta, \gamma>0$ satisfy the conditions
        \beq \frac{1}{\beta}=\frac{\theta}{\kappa}+\frac{1-\theta}{\gamma},\quad\quad (1-\theta)\beta=2.\label{eq:0.0}\eeq
    \end{lem}
    \begin{proof} Let $\theta$, $\kappa, \beta, \gamma>0$ be the constants satisfying (\ref{eq:0.0}).
        By Lemma \ref{lem:2.1}, for any $t\in[0,T)$, we have
        \beq
        \|u\|_{\beta,t}\leq\|u\|_{\kappa,t}^{\theta}\|u\|_{\gamma,t}^{1-\theta}.\no
        \eeq
        Now taking $\beta$-power and integrating with respect to $t$, we get
        \beqn
        \int_{0}^{T}\,dt\int_M|u|^{\beta}\,\oo_\varphi^n&\leq& \sup_{[0,T)}\|u\|_{\kappa,t}^{\theta\beta}
\int_{0}^{T}\|u\|_{\gamma,t}^{(1-\theta)\beta}\,dt\no\\
&=&\sup_{[0,T)}\|u\|_{\kappa,t}^{\theta\beta}
\int_{0}^{T}\|u\|_{\gamma,t}^{2}\,dt\label{eq:2.7}
        \eeqn
        By Theorem \ref{theo:2.1}, we have $$\|u\|_{\gamma,t}^2\leq C(\oo_g, n,\gamma,\|F\|_\infty)\int_M\Big(|\nabla u|^2+|u|^2\Big)\oo_\varphi^n.$$ Substituting this result into (\ref{eq:2.7}) and using the assumption (\ref{eq:0.0}), we have the inequality (\ref{eq:0.1}). The lemma is proved.
    \end{proof}

\textbf{Notations. } Throughout the paper, we always assume that $n, p, \ga,  \ka, \bb$ and $ \te$   satisfy the following conditions:
 \begin{enumerate}
   \item[(1).] The complex dimension $n$ of $M$ satisfies $n\geq 2$, and we assume $p>n;$
   \item[(2).] $\ga\in \Big(2, \frac {2n}{n-1}\Big)$ and we choose $\ga$  close to $\frac {2n}{n-1}$;
   \item[(3).] $\ka\in (0, 2)$ is small;
   \item[(4).] We define $\bb=2+\Big(1-\frac 2{\ga}\Big)\ka$. We can check that $\bb$ satisfies $2<\bb<\ga.;$
       \item[(5).]  $\te\in (0, 1)$ satisfies $(1-\te)\bb=2.$
 \end{enumerate}
We can check that the above assumptions imply that the constants $  \gamma, \kappa,\beta$ and $\te$ satisfy the conditions (\ref{eq:0.0}).

\section{Estimates }
    \subsection{The $L^\infty$ estimates}

    In this subsection, we use the parabolic version of Lu-Seyyedali \cite{[LS]} to  show that $\|\varphi\|_{\infty}$ and $\|F\|_{\infty}$ are bounded along the Calabi flow. To simplify the notations, we define the function
     $\Phi(s)=\sqrt{1+s^2}$ and we introduce  $Q_F$, $A_{R,p}$ and $B_{R,p}$ as follows:
\beqs
Q_F&=&\Big(\int_{0}^{T}dt\int_M\Phi(F)\,\oo_\varphi^n\Big)^\frac 1n,\\
A_{R,p}&=&\Big(\int_{0}^{T}dt\int_M\Phi(R)^{p}\,\oo_{\varphi}^n\Big)^\frac 1n,\\
B_{R,p}&=&\Big(\int_{0}^{T}dt\int_M\Phi(\Delta_\varphi R)^{p}\,\oo_{\varphi}^n\Big)^\frac 1n.
\eeqs

The main result of this subsection is the following theorem.

    \begin{theo}\label{theo:main1} Let $\varphi(x, t)(t\in [0, T))$ be the solution of Calabi flow (\ref{eq000}) with $T<\infty.$
     If $A_{R,p_1}$ and $B_{R,p_2}$ are bounded with $p_1> n+1$ and $p_2> n+1$, and $Q_F$ is also bounded. Then we have
      \beq
   \|\varphi\|_{L^{\infty}(M\times [0, T)) }+ \|F\|_{L^{\infty}(M\times [0, T))}\leq C(n,\oo_g,Q_F, A_{R,p_1},B_{R,p_2}, \varphi(0), T). \label{eq:0.5}
    \eeq
\end{theo}

First, we recall Chen-Cheng's result.

\begin{theo}(cf. \cite[Theorem 1.1 and Proposition 2.3]{[CC1]})\label{theo:3.1} Let $T>0$.  Consider the parabolic complex Monge-Ampere equation
\beqn
(-\partial_t\varphi)\;\oo_\varphi^n&=&e^H\oo_g^n,\label{eq:Z1}\\
\varphi(\cdot, 0)&=&\varphi_0. \label{eq:Z}
\eeqn
We have the following results.
\begin{enumerate}
    \item[(1).] Assume that $\varphi_0\in \cH(\oo_g)$   and $H(x, t)$ is smooth on $M\times [0, T]$. Then there exists a unique smooth
  solution $\varphi(x, t)$ to (\ref{eq:Z1})-(\ref{eq:Z}) on $M\times [0, T]$  starting from $\varphi_0$ such that
  $-\pd {\varphi}t>0$ and $\oo_g+\pbp\varphi(x, t)>0.$
  \item[(2).] If $H$ satisfies the condition
  \beq
  Ent_p(H):=\int_0^T\int_M\;e^H(|H|^p+1)\,\oo_g^n\,dt<\infty,\quad p>n+1,
  \eeq then we have
  \beq
  \|\varphi\|_{L^{\infty}}\leq C\Big(\oo_g, p, n, \|\varphi_0\|_{L^{\infty}}, T,  Ent_p(H)\Big).
  \eeq
\end{enumerate}

\end{theo}

The following result is proved by Lu-Seyyedali \cite{[LS]}, and we conclude the proof for completeness.
    \begin{lem}\label{lem:3.2}(cf. \cite[Lemma 2.1]{[LS]}  )
        Let $h:X\to\mathbf{R}$ be a positive smooth function and $\varphi$ and $v$ be K\"ahler potentials such that
        \beqn
        (\oo_g+\pbp\varphi)^n&=&e^F\oo_g^n,\no\\
        (\oo_g+\pbp v)^n&=&e^Fh^n\oo_g^n.
        \eeqn
        Then $\Delta_\varphi v\geq nh-\tr_\varphi \oo_g$.
        \end{lem}
    \begin{proof}
        We compute
        \beqn
        \Delta_\varphi v&=&\tr_\varphi(\pbp v)=\tr_\varphi(\oo_v-\oo_g)      \geq n\Big(\frac{\oo_v}{\oo_\varphi}\Big)^{\frac 1n}-\tr_\varphi\oo_g\no\\
        &\geq& n(e^{\frac Fn} h)e^{-\frac Fn}-\tr_\varphi\oo_g=nh-\tr_\varphi\oo_g .\no
        \eeqn
\end{proof}

The next result shows that $|\sup_M\varphi|$ is uniformly bounded along the Calabi flow.
\begin{lem}\label{Lem:3.4}(cf. \cite[Proof of Theorem 1.2]{[LZZ]})
   Let $\varphi(t)(t\in [0, T))$ be a solution of Calabi flow (\ref{eq000}) with $T<\infty.$   Then $|\sup_M\varphi|$ is  bounded by $\varphi(0)$ and $T$.
\end{lem}
\begin{proof} The proof is divided into several steps.

   (1).  Let $\psi(t)=\varphi(t+\frac T2)$. Then $\psi(t)(t\in[0,\frac T2)$ is the solution to the Calabi flow. According to \cite{[Cal3]} the distance $d_2(\varphi(t),\psi(t))$ is non-increasing for $t\in[0,\frac T2)$. Therefore,
    \beqn
    d_2(\varphi(t),\psi(t))\leq d_2(\varphi(0),\psi(0))=d_2(\varphi(0),\varphi(\frac T2)).
    \eeqn
    This implies that for any $t\in[\frac T2,T)$, we have
    \beqn
    d_2(\varphi(0),\varphi(t))&\leq& d_2(\varphi(0),\varphi(t-\frac T2))+d_2(\varphi(t-\frac T2),\varphi(t))\no\\
    &\leq&\max_{s\in[0,\frac T2]}d_2(\varphi(0),\varphi(s))+d_2(\varphi(0),\varphi(\frac T2)).\label{eq:3.1}
    \eeqn

(2). We show that $d_1(\varphi(0),\varphi(t))$ is bounded. Indeed, for any two K\"ahler potentials $\phi_0$, $\phi_1$ and any smooth path $\phi_s(s\in[0,1])$ connecting $\phi_0$ and $\phi_1$, we have
    \beqn
    L_1(\phi_0,\phi_1):=\int_{0}^{1}\|\phi_s\|_{L^1(\oo_{\phi_s)}}ds\leq \vol(\oo_g)^{\frac 12}\int_{0}^{1}\|\phi_s\|_{L^2(\oo_{\phi_s)}}ds:=\vol(\oo_g)^{\frac 12}L_2(\phi_0,\phi_1).\label{eq:3.7}
    \eeqn
    Taking the infimum with respect to all smooth path connecting $\phi_0$ and $\phi_1$,  we have
\beq
d_1(\phi_0, \phi_1)\leq \vol(\oo_{g})^{\frac 12} d_2(\phi_0, \phi_1). \label{eq:B13}
\eeq
Therefore, (\ref{eq:3.1}) and (\ref{eq:B13}) imply that $d_1(\varphi(0),\varphi(t))$ is bounded for $t\in[0,T).$

(3).
We show that $|\sup_M\varphi|$ is uniformly bounded for $t\in[0,T)$. Without loss of generality, we may assume that $\varphi(0)\in\mathcal{H}_0$. Then by the equality (\ref{eq:I}) we have
    \beqn
    \frac d{dt}I_\oo(\varphi(t))=\frac{1}{n!}\int_M\frac{\partial \varphi}{\partial t}\;\oo_{\varphi(t)}^n=\frac{1}{n!}\int_M(R-\underline{R})\oo_{\varphi(t)}^n=0.\label{eq:3.10}
    \eeqn
    Thus (\ref{eq:3.10}) shows that $\varphi(t)\in\mathcal{H}_0$ for all $t\in[0,T).$ According to the Lemma 4.4 in Chen-Cheng\cite{[CC2]}, we have
    \beqn
    |\sup_M\varphi|\leq C\Big(d_1(0,\varphi)+1\Big)\leq C\Big(d_1(0,\varphi(0))+d_1(\varphi(0),\varphi(t))\Big)\label{eq:3.9}
\eeqn
     for some constant $C$. Combining with (\ref{eq:3.7}) and (\ref{eq:3.9}), we conclude this lemma.
\end{proof}

Combining the above results, we show that the space-time integral of $e^F$ is bounded for some $q>1.$

\begin{lem}\label{theo:main2}
        Let $\varphi(t)(t\in [0, T))$ be a solution of Calabi flow (\ref{eq000}) with $T<\infty.$   If $A_{R,p_1}$ and $B_{R,p_2}$ are bounded  with $\min\{p_1,p_2\}> n+1$ and $Q_F$ is uniformly bounded, then there exist $\delta_0>0$ and $C$ depending on $n,\oo_g,Q_F, A_{R,p_1}$, $B_{R,p_2}$, $\varphi(0)$ and $T$ such that
        \beq
        \int_{0}^{T}\,dt\int_{M}\;e^{(1+\delta_0)F}\,\oo_g^n\leq C.\label{eq:0.2}
        \eeq
    \end{lem}
    \begin{proof}
         We construct auxiliary functions $\psi$, $\rho$ and $v$  as the solutions to the following equations:
        \beqn
        &&(-\partial_t\psi)\oo_{\psi}^n=Q_F^{-n}\Phi(F)e^F\oo_g^n; \ \     \psi\Big|_{t=0}=0,\no\\
        &&(-\partial_t\rho)\oo_{\rho}^n=A_{R,p_1}^{-n}\Phi(R)^{p_1}e^F\oo_g^n;\ \  \rho\Big|_{t=0}=0,\no\\
        &&(-\partial_t v)\oo_{v}^n=B_{R,p_2}^{-n}\Phi(\Delta_\varphi R)^{p_2}e^F\oo_g^n,\ \  v\Big|_{t=0}=0.
        \eeqn
        Note that the existence of $\psi,\rho,v$ is guaranteed by Theorem \ref{theo:3.1}. For $0<\ee \leq 1$, we define $$u=F+\ee\psi+\ee\rho +\ee v-\lambda\varphi. $$ Using Lemma \ref{lem:3.2}, we can compute
        \beqn
        &&e^{-\delta u}(\Delta_\varphi-\partial_t)(e^{\delta u})\geq \delta\Delta_\varphi u-\delta \dot{u}\no\\
        &&\geq\delta(-R+\tr_{\varphi}\Rc)+\ee\delta\Big(nQ_F^{-1}(-\dot{\psi})^{-\frac 1n}\Phi(F)^{\frac 1n}-\tr_\varphi \oo_g\Big)\no\\
        &&\quad+\ee\delta\Big(nA_{R,p_1}^{-1}(-\dot{\rho})^{-\frac 1n}\Phi(R)^{\frac {p_1}{n}}-\tr_\varphi\oo_g\Big)+\ee\delta\Big(nB_{R,p_2}^{-1}(-\dot{v})^{-\frac 1n}\Phi(\Delta_\varphi R)^{\frac {p_2}{n}}-\tr_\varphi \oo_g\Big)\no\\
        &&\quad-n\lambda\delta+\delta\lambda \tr_\varphi\oo_g+\delta\Big(-\Delta_\varphi R-\ee\dot\psi-\ee\dot\rho-\ee \dot v+\lambda\dot\varphi\Big),\label{eq:3.14}
        \eeqn
     where we write $\dot u=\partial_t u$ for short. Choosing $\lambda=3+|\Rc|_g$ in (\ref{eq:3.14}), we have
    \beqn
    &&e^{-\delta u}(\Delta_\varphi-\partial_t)(e^{\delta u})\no\\&&\geq \delta\Big(-R+\ee nA_{R,p_1}^{-1}(-\dot\rho)^{-\frac 1n}\Phi(R)^{\frac {p_1}{n}}-\ee\dot\rho\Big)+\delta\Big(-\Delta_\varphi R\no\\
    &&\quad+\ee nB_{R,p_2}^{-1}(-\dot v)^{-\frac 1n}\Phi(\Delta_\varphi R)^{\frac {p_2}{n}}-\ee \dot v\Big)+\delta\ee\Big(nQ_F^{-1}(-\dot\psi)^{-\frac 1n}\Phi(F)^{\frac 1n}-\dot\psi\Big)\no\\
        &&\quad-n\lambda\delta+\delta\lambda(R-\underline{R})\no\\
        &&\geq\delta\Big((\lambda-1)R+\ee A_{R,p_1}^{-1}(-\dot\rho)^{-\frac 1n}\Phi(R)^{\frac {p_1}{n}}-\ee\dot\rho\Big)+\delta\Big(-\Delta_\varphi R\no\\
        &&\quad+\ee B_{R,p_2}^{-1}(-\dot v)^{-\frac 1n}\Phi(\Delta_\varphi R)^{\frac {p_2}{n}}-\ee \dot v\Big)\no\\
        &&\quad+\delta\ee\Big(nQ_F^{-1}(-\dot\psi)^{-\frac 1n}\Phi(F)^{\frac 1n}-\dot\psi\Big)-C,
    \eeqn
    where $C=n\lambda\delta+\delta\lambda\underline{R}$. Since $Bx^{-\frac 1n}+x\geq C(n)B^{\frac{n}{n+1}}$ for all $x> 0$,
     we get
    \beqn
        e^{-\delta u}(\Delta_\varphi-\partial_t)(e^{\delta u})&\geq&\delta\Big((\lambda-1)R+C\ee\Phi(R)^{\frac{p_1}{n+1}}\Big)+\delta\Big(-\Delta_\varphi R+C\ee\Phi(\Delta_\varphi R)^{\frac{p_2}{n+1}}\Big)\no\\
        &\quad&+\delta\ee C\Phi(F)^{\frac{1}{n+1}}-C
    \eeqn
   where $C$ depends on $n,A_{R,p_1}$, $B_{R,p_2}$ and $Q_F$. Let $\hat{\Phi}(F)=\delta\ee C\Phi(F)^{\frac{1}{n+1}}$. As a result, we have
    \beqn
    \int_{0}^{T}\;dt\int_{M}\;(\Delta_\varphi-\partial_t)(e^{\delta u})\oo_{\varphi}^n&\geq&
    \int_{0}^{T}\;dt\int_{M}\;e^{\delta u}\Big(\delta\Big((\lambda-1)R+C\ee\Phi(R)^{\frac{p_1}{n+1}}\Big)\no\\
    &\quad&+\delta\Big(-\Delta_\varphi R+C\ee\Phi(\Delta_\varphi R)^{\frac{p_2}{n+1}}\Big)
    +\hat{\Phi}(F)-C\Big)\oo_{\varphi}^n.\no\\
    \quad\label{eq:3.17}
    \eeqn
Using the equation of Calabi flow, we have
    \beqn
    \int_{0}^{T}\;dt\int_{M}\;(\Delta_\varphi-\partial_t)(e^{\delta u})\oo_{\varphi}^n&=&\int_{0}^{T}\;dt\int_{M}-\partial_te^{\delta u}\;\oo_{\varphi}^n\no\\
&&=\int_{0}^{T}-\partial_t\Big(\int_Me^{\delta u}\oo_{\varphi}^n\Big)\;dt+\int_{0}^{T}\;dt\int_M\;e^{\delta u}\partial_t(\oo_{\varphi}^n)\no\\
&&\leq\int_Me^{\delta u}\oo_{\varphi}^n\Big|_{t=0}+\int_{0}^{T}\;dt\int_M\;e^{\delta u}\dot{F}\;\oo_{\varphi}^n,\label{eq:3.18}
    \eeqn
   Combining (\ref{eq:3.17}) with (\ref{eq:3.18}), we get
    \beqn
    \int_Me^{\delta u}\oo_{\varphi}^n\Big|_{t=0}&\geq&\int_{0}^{T}\;dt\int_{M}e^{\delta u}\Big(\delta\Big((\lambda-1)R+C\ee\Phi(R)^{\frac{p_1}{n+1}}\Big)\no\\
    &\quad&+\delta\Big(-(1+\frac{1}{\delta})\Delta_\varphi R+C\ee\Phi(\Delta_\varphi R)^{\frac{p_2}{n+1}}\Big)
    +\hat{\Phi}(F)-C\Big)\,\oo_{\varphi}^n.
    \eeqn
    Since $Cx^{\beta}-x$ has lower bound which is independent of $x$ for all $\beta> 1$ and $\min\{p_1,p_2\}> n+1$, we get
    \beqn
        \int_Me^{\delta u}\oo_{\varphi}^n\Big|_{t=0}\geq\int_{0}^{T}\;dt\int_{M}\;e^{\delta u}\Big(\hat{\Phi}(F)-C(\lambda,\delta,\ee,n,A_{R,p_1},B_{R,p_2},Q_F,\oo_g)\Big)\oo_{\varphi}^n.\label{eq:3.19}
    \eeqn
Choosing $\delta=\lambda^{-1}\alpha(M,\omega_g)$, where $\alpha(M,\omega_g)$ is the $\alpha$ invariant of $\oo_g$, we have that
    \beqn
    \int_{0}^{T}\;dt\int_{M}\;e^{\delta u}\Big(\hat{\Phi}(F)-C\Big)\;\oo_{\varphi}^n\leq C(\delta,\lambda,\varphi(0)).\label{eq:3.20}
    \eeqn
     Next we define
    \beqn
    E_1&=&\{(x,t)\in M\times[0,T):\hat{\Phi}(F)-C\geq1\},\no\\
    E_2&=&\{(x,t)\in M\times[0,T):\hat{\Phi}(F)-C< 1\}.
    \eeqn
By definition, $F$ is bounded on $E_2$. Thus by (\ref{eq:3.19}) and  (\ref{eq:3.20})   we have
    \beqn
    \int_{E_1}e^{\delta u}\;\oo_{\varphi}^n\,dt&\leq&   \int_{E_1}e^{\delta u}\Big(\hat{\Phi}(F)-C\Big)\;\oo_{\varphi}^n\,dt\no\\
    &\leq&C -\int_{E_2}e^{\delta u}\Big(\hat{\Phi}(F)-C\Big)\;\oo_{\varphi}^n\,dt\no\\
    &\leq&C+C\int_{E_2}e^{\delta u}\,\oo_{\varphi}^n\,dt\no\\
    &\leq& C+C\int_{E_2}e^{\delta F-\lambda\delta  \varphi}\;\oo_{\varphi}^n\,dt\no\\
    &\leq& C(n,\delta,\ee,\lambda,A_{R,p_1},B_{R,p_2},Q_F,\oo_g,\varphi(0),T).\label{eq:A1}
    \eeqn
    By definition of $u$, we have
    \beqn
    \int_{E_1}e^{(1+\delta)F+\ee\delta(\psi+\rho+v)}\,\oo_g^n\,dt\leq e^{\delta\lambda|\sup_{M}\varphi|}\int_{E_1}e^{\delta u+F}\,\oo_g^n\,dt.
    \eeqn
 Since $|\sup_M\varphi|$ is bounded by Lemma \ref{Lem:3.4}, we conclude that    $\int_{E_1}e^{(1+\delta)F+\ee\delta(\psi+\rho+v)}\oo_g^ndt$ is bounded. Using H\"older inequality, we get
    \beqn
    &&\int_{E_1}e^{(1+\frac{\delta}{2})F}\oo_g^ndt=\int_{E_1}e^{(1+\frac{\delta}{2})F+\frac{1+\frac{\delta}{2}}{1+\delta}\ee\delta(\psi+\rho+v)}e^{-\frac{1+\frac{\delta}{2}}{1+\delta}\ee\delta(\psi+\rho+v)}\oo_g^ndt\no\\
    &\leq&\Big(\int_{E_1}e^{(1+\delta)F+\ee\delta(\psi+\rho+v)}\oo_g^ndt\Big)^{\frac{1+\frac{\delta}{2}}{1+\delta}}\Big(\int_{E_1}e^{-\frac{1+\frac{\delta}{2}}{\frac{\delta}{2}}\ee\delta(\psi+\rho+v)}\oo_g^ndt\Big)^{\frac{\frac{\delta}{2}}{1+\delta}}.
    \eeqn
    Choosing $\ee$ small enough such that $(2+\delta)\ee<\frac{\alpha(M,\oo_g)}{3}$,   we conclude that
    \beqn
    \int_{E_1}e^{(1+\frac{\delta}{2})F}\oo_g^ndt\leq C(n,\lambda,\delta,\ee,A_{R,p_1},B_{R,p_2},Q_F,\oo_g,\varphi(0),T). \label{eq:0.3}
    \eeqn
    Combining (\ref{eq:0.3}) with the fact that $F$ is bounded on $E_2$, we have (\ref{eq:0.2}). The lemma is proved.
\end{proof}

Using Lemma \ref{theo:main2} and the Calabi flow equation, we show that the $L^{q}(M, \oo_g)$ norm of $F$ is bounded for some $q>1.$
\begin{lem}\label{lem:interC} Under the assumption of Theorem \ref{theo:main1}, there exist $\delta_1$ and $C$ depending on $n,\oo_g,Q_F, A_{R,p_1}$, $B_{R,p_2}$, $\varphi(0)$ and $T$ such that
    \beqn
    \int_M e^{\delta_1 F}\,\oo_\varphi^n\leq C.\label{eq:0.4}
    \eeqn
\end{lem}
\begin{proof}
    Let $\delta> 0$.  Taking the derivative with respect to $t$, we find that
    \beqn
    \frac{\partial}{\partial t}\Big(\int_M\;e^{\delta F}\oo_{\varphi}^n\Big)=\int_M(1+\delta)\dot{F}e^{\delta F}\;\oo_{\varphi}^n. \no
    \eeqn
    Hence we have
    \beqn
    \int_M\;e^{\delta F}\;\oo_{\varphi}^n\Big|_t-\int_M\;e^{\delta F}\;\oo_{\varphi}^n\Big|_0\leq\int_{0}^{T}\;dt\int_{M}(1+\delta)|\dot{F}|e^{\delta F}\oo_{\varphi}^n. \no
    \eeqn
    Using the H\"older inequality we have
    \beqn
    \int_{0}^{T}\;dt\int_{M}\;(1+\delta)\dot{F}e^{\delta F}\;\oo_{\varphi}^n\leq(1+\delta)\Big(\int_{0}^{T}\;dt\int_{M}\;e^{l\delta F}\;\oo_{\varphi}^n\Big)^{\frac 1l}\Big(\int_{0}^{T}\;dt\int_{M}\;|\Delta_\varphi R|^{p_2}\;\oo_{\varphi}^n\Big)^{\frac {1}{p_2}},\no
    \eeqn
    where $\frac1l+\frac {1}{p_2}=1$. Choosing $\delta$ small and using Lemma  \ref{theo:main2}, we have (\ref{eq:0.4}). The lemma is proved.
\end{proof}
Combining the above estimates and using the maximum principles, we show Theorem \ref{theo:main1}.
\begin{proof}[Proof of Theorem \ref{theo:main1}]
    By Theorem \ref{theo:3.1} and Lemma \ref{theo:main2}, we conclude that $\psi$ is bounded. Moreover by Lemma \ref{lem:interC} we conclude that $\varphi$ is also uniformly bounded. We define new auxiliary functions as the solutions of the following equations:
    \beqn
    (-\partial_t\rho)\oo_{\rho}^n&=&A_{R,q}^{-n}\Phi(R)^qe^F\oo_g^n,\ \  \rho\Big|_{t=0}=0,\no\\
    (-\partial_t v)\oo_{v}^n&=&B_{R,q}^{-n}\Phi(\Delta_\varphi R)^qe^F\oo_g^n,\ \  v\Big|_{t=0}=0,
    \eeqn
    where $n+1< q< \min\{p_1,p_2\}$. For $0<\sigma< \delta_0$, we have
    \beqn
    &&\int_{0}^{T}\;dt\int_{M}\;|\Phi(R)|^{(1+\sigma)q}e^{(1+\sigma)F}\;\oo_g^n=\int_{0}^{T}\;dt\int_{M}\;|\Phi(R)|^{(1+\sigma)q}e^{\sigma F}\;\oo_{\varphi}^n\no\\
    &\leq&\Big(\int_{0}^{T}\;dt\int_{M}\;|\Phi(R)|^{(1+\sigma)q\frac{\delta_0}{\delta_0-\sigma}}\;\oo_{\varphi}^n\Big)^{\frac{\delta_0-\sigma}{\delta_0}}\Big(\int_{0}^{T}\;dt\int_{M}\;e^{\delta_0F}\;\oo_\varphi^n\Big)^{\frac {\sigma}{\delta_0}}\no\\
    &\leq& C(n,\oo_g,Q_F,A_{R,p_1},B_{R,p_2},\varphi(0),T)
    \Big(\int_{0}^{T}\;dt\int_{M}\;|\Phi(R)|^{(1+\sigma)q
    \frac{\delta_0}{\delta_0-\sigma}}\oo_{\varphi}^n\Big)^{\frac{\delta_0-\sigma}{\delta_0}},\no
    \eeqn
where we used Lemma \ref{theo:main2} in the last inequality. Now we can choose $\sigma$ small enough such that $(1+\sigma)\frac{\delta_0}{\delta_0-\sigma}q< p_1$. Therefore, we conclude that $\rho$ is bounded by Theorem \ref{theo:3.1}. Similarly, we have that $v$ is also bounded. Let $u=F+\psi+\rho+v-\lambda \varphi$, we have
    \beqn
    (\Delta_\varphi-\partial_t)u\geq e^{\delta u}\Big(\hat{\Phi}(F)-C\Big),
    \eeqn
    where we use the same argument as  in the proof of Theorem \ref{theo:main2} and $C$ depends on $n,\oo_g,A_{R,p_1},B_{R,p_2}$ and $Q_F$. Fixing $\ee> 0$, we denote $(x_0,t_0)$ the maximum point of $u$ on $M\times[0,T-\ee]$. We have
   \beq
    0\geq   (\Delta_\varphi-\partial_t)u\geq e^{\delta u}\Big(\hat{\Phi}(F)-C\Big).\no
    \eeq
    This implies that $|F(x_0,t_0)|$ is bounded. As a result, for any $(x,t)\in M\times [0,T-\ee]$
    \beqn
    u(x,t)\leq u(x_0,t_0)&=&F(x_0,t_0)+(\psi+\rho+v-\lambda\varphi)((x_0,t_0)\no\\&\leq& C(n,\oo_g,Q_F,A_{R,p_1},B_{R,p_2},\varphi(0),T).\no
    \eeqn
    This implies $F\leq C$.  Replacing $u$ by $u'=-F+\psi+\rho+v-\lambda\varphi$, the same argument shows that $F\geq -C$. Therefore we conclude that on $M\times[0,T-\ee]$,
    \beq
    |F|\leq C(n,\oo_g,Q_F, A_{R,p_1},B_{R,p_2}, \varphi(0), T).\no
    \eeq
        Taking $\ee\ri 0$, we have (\ref{eq:0.5}). The theorem is proved.
\end{proof}
\subsection{Estimates of $\|n+\Delta\varphi\|_{s}$}
In this subsection, we  follow similar method as in Chen-Cheng \cite{[CC3]} and Li-Zhang-Zheng \cite{[LZZ]} to prove that $\|n+\Delta\varphi\|_{s}$ is bounded if $Q_F,A_{R,2p}^n,B_{R,p+1}^n$ are bounded with $p> n$. We recall the following Chen-Cheng's estimates in \cite{[CC3]}, see also Li-Zhang-Zheng \cite{[LZZ]}.

\begin{lem}\label{lem:interD}(cf. \cite{[CC3]}, \cite[Lemma 2.3]{[LZZ]}) We define
        \beq
    v=e^{-\alpha(F+\lambda\varphi)}(n+\Delta_g\varphi). \label{eq:v3}
    \eeq
    Let $q>1$ and $\alpha\geq q$. There exists a constant $C(\oo_g)$ such that for $\la>C(\oo_g)$, we have
    \beqn &&
    \frac {3(q-1)}{q^2}\int_M\;  |\Na  v^{\frac q2}|_{\varphi}^2\;\oo_{\varphi}^n+\frac {\la \al}4\int_M\; e^{\frac \al{n-1} (F+\la \varphi)-\frac F{n-1}} v^{q+\frac{1}{n-1}}
    \,\oo_{\varphi}^n \nonumber\\
    &\leq & \int_M\; \td Rv^{q}\oo_{\varphi}^n,\label{eq:v12}
    \eeqn where $\td R=\al(\la n-R)+\frac {\al\la}{\al-1}+\frac 1n e^{-\frac Fn}R_g$ .
\end{lem}

Using the equation (\ref{eq000}) of Calabi flow, we have
\begin{lem}\label{lem:3.7}
    Let $   v=e^{-\alpha(F+\lambda\varphi)}(n+\Delta_g\varphi)$ as in Lemma \ref{lem:interD}. For any $q> 0$, we have
    \beqn
    \int_M\; v^{q}\;\oo_{\varphi}^n\Big|_t-\int_M\; v^{q}\;\oo_{\varphi}^n\Big|_0&\leq& Cq\Big(\int_{0}^{T}\;dt\int_{M}\;v^{qr}\;\oo_{\varphi}^n\Big)^{\frac 1r}+Cq\Big(\int_{0}^{T}\;dt\int_{M}\;v^{qb}\;\oo_{\varphi}^n\Big)^{\frac 1b}\no\\\ &&+Cq\Big(\int_{0}^{T}\;dt\int_{M}\;v^{2q}\;\oo_{\varphi}^n\Big)^{\frac 12}, \label{eq:0.6}
    \eeqn
where C depends on $\al, A_{R,2p}^n, B_{R,p+1}^n, \|\varphi\|_\infty$ and $\|F\|_\infty$. Moreover, $p,r$ and $b$ satisfy the following conditions:
\beq
\frac{1}{2p}+\frac1r=1,\quad\quad
 \frac1{p+1}+\frac 1b=1.\label{eq:2.4}
 \eeq
\end{lem}
\begin{proof}
    Taking the derivative with respect to $t$, we get
    \beqn
    \frac{\partial}{\partial t}\Big(\int_M\; v^{q}\;\oo_\varphi^n\Big)=\int_M \;\Big(q v^{q-1}\dot{v}+v^{q}\Delta_{\varphi} R\Big)\;\oo_{\varphi}^n\label{eq:Q}.
    \eeqn
    Putting $\dot{v}=-\alpha(\dot{F}+\lambda\dot{\varphi})v+e^{-\alpha(F+\lambda\varphi)}\Delta_g R$    into (\ref{eq:Q}), we have
    \beqn
    \frac{\partial}{\partial t}\Big(\int_M \;v^{q}\;\oo_\varphi^n\Big)&=&\int_M\;\Big( -\alpha qv^q(\dot{F}+\lambda\dot{\varphi})+qv^{q-1}e^{-\alpha(F+\lambda\varphi)}\Delta_g R+v^q\Delta_{\varphi}R\Big)\;\oo_{\varphi}^n\no\\
&=&\int_M\;\Big( (1-\alpha q)v^q\Delta_{\varphi}R-\alpha qv^q\lambda(R-\un R)+qv^{q-1}e^{-\alpha(F+\lambda\varphi)}\Delta_g R\Big)\;\oo_{\varphi}^n.\no\\\label{eq:2.5}
   \eeqn
  Let  $p,r$ and $b$ be the constants satisfying (\ref{eq:2.4}).
Integrating both sides of (\ref{eq:2.5}) with respect to $t$ and using the H\"older inequality, we have
\beqn
\int_M \;v^{q}\oo_{\varphi}^n\Big|_t-\int_M v^{q}\;\oo_{\varphi}^n\Big|_0&\leq& Cq\Big(\int_{0}^{T}\;dt\int_{M}\;v^{qr}\;\oo_{\varphi}^n\Big)^{\frac 1r}+Cq\Big(\int_{0}^{T}\;dt\int_{M}\;v^{qb}\;\oo_{\varphi}^n\Big)^{\frac 1b}\no\\&+&Cq\int_{0}^{T}\;dt\int_{M}\;v^{q-1}|\Delta_g R|\;\oo_{\varphi}^n, \label{eq:A3}
\eeqn
where $C$ depends on $\al, \|\varphi\|_\infty,\|F\|_\infty, A_{R,2p}$ and $B_{R,p+1}$. Using the inequality  $|\Delta_g R|\leq |\nabla^2R|_{\varphi}(n+\Delta_g\varphi)$, we have
\beqn
\int_{0}^{T}dt\int_{M}v^{q-1}|\Delta_g R|\,\oo_{\varphi}^n\leq C(\|\varphi\|_\infty,\|F\|_\infty) \Big(\int_{0}^{T}dt\int_{M}v^{2q}\;\oo_{\varphi}^n\Big)^{\frac 12}\Big(\int_{0}^{T}dt\int_{M}|\nabla^2R|_\varphi^2\,\oo_{\varphi}^n\Big)^{\frac 12}.\no\\
\label{eq:34}
\eeqn
Note that \beq \int_{0}^{T}\;dt\int_{M}\;|\nabla^2R|_{\varphi}^2\;\oo_{\varphi}^n=
\int_{0}^{T}\;dt\int_{M}\;|\Delta_{\varphi}R|^2\;\oo_{\varphi}^n\label{eq:A4}\eeq and $B_{R,p+1}^n$ is bounded with $p> n+1\geq 2$. Combining (\ref{eq:A3})-(\ref{eq:A4}),  we have the inequality (\ref{eq:0.6}).
\end{proof}

Combining Lemma \ref{lem:interD}, Lemma \ref{lem:3.7} with Lemma \ref{lem:interF}, we have the result.
\begin{lem}\label{theo:main3}
Under the assumption that $Q_F,A_{R,2p}$ and $B_{R,p+1}$ are bounded with $p> n$, for any $s\geq 1$, there exists a constant $C$ depending on $n,s,\oo_g,Q_F,A_{R,2p},B_{R,p+1}$ and $\varphi(0)$ and $T$ such that
    \beq
    \int_{0}^{T}\;dt\int_{M}\;(n+\Delta_g \varphi)^s\;\oo_{\varphi}^n\leq C. \label{eq:0.7}
    \eeq
\end{lem}
\begin{proof}
    By Lemma \ref{lem:interF} and Lemma \ref{lem:interD}, for any $q> 1$ we have
    \beqn
    &&\int_{0}^{T}dt\int_Mv^{\frac{\beta q}{2}}\;\oo_{\varphi}^n\no\\&& \leq C(n,\oo_g,\|F\|_\infty,\gamma)\sup_{[0,T)}\|v^{\frac q2}\|_{\kappa,t}^{(1-\frac 2\gamma)\kappa}\int_{0}^{T}dt\int_M(|\nabla v^{\frac q2}|_\varphi^2+|v|^q)\;\oo_{\varphi}^n\no\\
     &&\leq C\frac{q^2}{3(q-1)}\sup_{t\in[0,T)}\|v^{\frac q2}\|_{\kappa,t}^{(1-\frac 2\gamma)\kappa}\int_{0}^{T}dt\int_M(\td R+1)v^q\,\oo_{\varphi}^n\no\\
    &&\leq C(n,\oo_g,\|F\|_\infty,\gamma,A_{R,2p})\frac{q^2}{3(q-1)}\sup_{t\in[0,T)}\|v^{\frac q2}\|_{\kappa,t}^{(1-\frac 2\gamma)\kappa}\Big(\int_{0}^{T}\;dt\int_M\;v^{qr}\;\oo_{\varphi}^n\Big)^{\frac 1r},\no\\\label{eq:B1}
\eeqn where $r$ is the constant satisfying (\ref{eq:2.4}).
Note that $\|v^{\frac q2}\|_{\kappa,t}^{(1-\frac 2\gamma)\kappa}=\|v\|_{\frac {q\kappa}2,t}^{(\frac12-\frac1\gamma)\kappa q}$.  By Lemma \ref{lem:3.7} we have
\beqn
\|v^{\frac q2}\|_{\kappa,t}^{(\frac12-\frac1\gamma)\kappa }&\leq& \Big(\int_M\;v^{\frac{q\kappa}{2}}\;\oo_{\varphi}^n\Big|_{t=0}+\frac{Cq\kappa}{2}\Big(\int_{0}^{T}\;dt\int_{M}\;v^{\frac{q\kappa r}{2}}\;\oo_{\varphi}^n\Big)^{\frac 1r}\no\\&\quad&+\frac{Cq\kappa}{2}\Big(\int_{0}^{T}\;dt\int_{M}\;v^{\frac{q\kappa b}{2}}\;\oo_{\varphi}^n\Big)^{\frac 1b}+\frac{Cq\kappa}{2}\Big(\int_{0}^{T}\;dt\int_{M}\;v^{q\kappa}\;\oo_{\varphi}^n\Big)^{\frac 12}\Big)^{1-\frac 2\gamma}.\no
\eeqn
Taking the $\frac{\beta q}{2}$-root in (\ref{eq:B1}), we have
\beqn &&
\|v\|_{\frac{\beta q}{2}}\leq C^{\frac{2}{\beta q}}\Big(\frac{q^2}{3(q-1)}\Big)^{\frac{2}{\beta q}}\Big(C
+\frac{Cq\kappa}{2}\Big(\int_{0}^{T}\;dt\int_{M}\;v^{\frac{q\kappa r}{2}}\;\oo_{\varphi}^n\Big)^{\frac 1r}
\no\\&\quad&+\frac{Cq\kappa}{2}\Big(\int_{0}^{T}\;dt\int_{M}\;v^{\frac{q\kappa b}{2}}\;\oo_{\varphi}^n\Big)^{\frac 1b}+\frac{Cq\kappa}{2}\Big(\int_{0}^{T}\;dt\int_{M}\;v^{q\kappa}\;\oo_{\varphi}^n\Big)^{\frac 12}\Big)^{\frac{2\theta}{q\kappa}}\|v\|_{qr}^{\frac{2}{\beta}},
\eeqn
where $C$ depends on $\al, n,\oo_g,\|\varphi\|_\infty,\|F\|_\infty,A_{R,2p},B_{R,p+1},\gamma$ and $\varphi(0)$. Since $p> n$, we have that $r=\frac{2p}{2p-1}< \frac{2n}{2n-1}< 2$ and $b=\frac{p+1}{p}<2$.  We choose $\beta$ and $\kappa$ such that
\beqn
\frac{\beta}{2}> \max\{\kappa,r\},\label{eq:3.41}
\eeqn
or equivalently,
\beqn
\frac{2r-2}{1-\frac 2\gamma}<\kappa<\frac{2}{1+\frac 2\gamma}.\label{eq:B8}
\eeqn
Since $r< \frac{2n}{2n-1}$, we can choose $\gamma$ close to $\frac{2n}{n-1}$ such that $\frac{2r-2}{1-\frac 2\gamma}<\frac{2}{1+\frac 2\gamma}$. For such $\kappa$, $\gamma$ and large $q$ with $q\kappa> 1$, we have
\beqn
\|v\|_{\frac{\beta q}{2}}&\leq&C^{\frac{2}{\beta q}}\Big(\frac{q^2}{3(q-1)}\Big)^{\frac{2}{\beta q}}\Big(C+Cq\kappa\|v\|_{{q\kappa}}^{\frac{q\kappa}{2}}\Big)^{\frac{2\theta}{q\kappa}}\|v\|_{{qr}}^{\frac{2}{\beta}}\no\\
&\leq&C^{\frac{2}{\beta q}}\Big(\frac{q^2}{3(q-1)}\Big)^{\frac{2}{\beta q}}C^{\frac{2\theta}{q\kappa}}(q\kappa)^{\frac{2\theta}{q\kappa}}
\|v\|_{{q\max\{r,\kappa\}}},
\eeqn
where $C$ depends on $\al, \oo_g,\kappa,\gamma,\|\varphi\|_\infty,\|F\|_\infty,
\varphi(0),A_{R,2p},B_{R,p+1}$ and in the last inequality we used the fact that
\beqs
v=e^{-\alpha(F+\lambda\varphi)}(n+\Delta_g\varphi)\geq C(\alpha,\|\varphi\|_\infty,\|F\|_\infty)\frac{1}{n}e^{\frac{F}{n}}\geq C(n,\alpha,\|\varphi\|_\infty,\|F\|_\infty).
\eeqs
Let $\al=2q.$ By the iteration argument there exists $q_0> 1$ such that for any $q> q_0$ we have
\beqn
\|v\|_{q}\leq C(n,\oo_g,q,\kappa,\gamma,\|F\|_\infty,\|\varphi\|_\infty,A_{R,2p},B_{R,p+1},\varphi(0))\|v\|_{q_0}.\label{eq:0.8}
\eeqn
Since $\|v\|_{q_0}\leq\ee\|v\|_{q}+C(\ee)\|v\|_{1}$, we have
$
\|v\|_{q}\leq C\|v\|_{1}
$
for small $\ee$. Now
\beqs
\|v\|_{1}&=&\int_{0}^{T}dt\int_M\;e^{-\alpha(F+\lambda\varphi)}
(n+\Delta_g\varphi)\;\oo_{\varphi}^n\\&\leq & C(q,\|F\|_\infty,\|\varphi\|_\infty)\int_{0}^{T}\;dt\int_M(n+\Delta_g\varphi)\;\oo_g^n
\\&\leq& C(n, q,\|\varphi\|_\infty,\|F\|_\infty,T).\no
\eeqs Combining this with (\ref{eq:0.8}), we have the inequality (\ref{eq:0.7}).
The lemma is proved.
\end{proof}

\subsection{Estimates of $\|\nabla F\|$}
In this subsection we show that the space-time norm $\|\nabla F\|_{{2s}}$ is bounded for any $s< 2p $. Note that we assumed the condition that $p>n$ in the assumption of Theorem \ref{theo:1.1}. The argument uses the parabolic Moser iteration and is based on Chen-Cheng \cite{[CC3]} and Li-Zhang-Zheng \cite{[LZZ]}.

\begin{lem}\label{lem:F}
    Under the assumption of Lemma \ref{theo:main3}, for any $s<2p$ there exists a constant $C$ depending on $n,s,\oo_g,Q_F,A_{R,2p},B_{R,p+1},\varphi(0)$ and $T$ such that
    \beqn
    \int_{0}^{T}\;dt\int_{M}\;|\Na F|_{\varphi}^{2s}\;\oo_{\varphi}^n \leq C.
    \eeqn

\end{lem}

Before proving Lemma \ref{lem:F}, we collect the conditions on the constants. Let $n, p, \ga,  \ka, \bb$ and $ \te$ be the constants in Section 2, and $r, b$  the constants
 in the proof of  Lemma \ref{lem:3.7} and Lemma \ref{theo:main3}  satisfying
 \beq
\frac 1{2p}+\frac 1r=1,\quad \frac 1b+\frac 1{p+1}=1.  \label{eq:3.48}
\eeq
We define the constants $a, d$ by
\beqn
&&\frac 1a+\frac 1{2p}+\frac 1{p+1}=1,
 \quad \frac 1d+\frac 1p=1.\label{eq:3.49}
\eeqn

To show Lemma \ref{lem:F}, we first estimate the $L^q$ norm of $|\Na F|$ on each time slices by using the equation (\ref{eq000}) of Calabi flow.

\begin{lem}\label{lem:G}
    Let $w=e^{\frac F2}|\Na F|_{\varphi}^2+1$ and $z=w^q$ with $q> \frac 1\kappa$. We have
    \beqn &&
\|z\|_{\kappa,t}^{\kappa}-\|z\|_{\kappa,0}^\kappa\leq Cq\kappa\|z\|_{b\kappa}^\kappa+Cq\kappa(q\kappa -\frac12)\Big(\frac{2q\kappa}{q\kappa -1}\Big)^{\frac12}\|z\|_{\frac{d(2q\kappa -2)}{q}}^{\frac{q\kappa -1}{q}}\no\\&&+Cq\kappa(q\kappa -\frac12)(q\kappa-1)^{-\frac12}\|z\|_{\frac{r(2q\kappa-1)}{q}}^{\frac{2q\kappa-1}{2q}}+Cq\kappa\|z\|_{\frac{a(q\kappa -1)}{q}}^{\frac{q\kappa -1}{q}}+Cq\kappa\|z\|_{2\kappa}^\kappa,
    \eeqn
where $C$ only depends on $n,p,\oo_g,Q_F,A_{R,2p}$, $B_{R,p+1},\varphi(0)$ and   $T$. Here $p$,$b$,$r$,$a$ and $d$ satisfy the   conditions (\ref{eq:3.48})-(\ref{eq:3.49}).

\end{lem}
\begin{proof}
Taking the derivative with respect to $t$ and using (\ref{eq000}), we get
\beqn
\frac{\partial}{\partial t}\|z\|_{\kappa,t}^{\kappa}=\frac{\partial}{\partial t}\int_{M}\;w^{\kappa q}\;\oo_{\varphi}^n=\int_{M}\;\Big(\kappa qw^{\kappa q-1}\dot{w}+w^{\kappa q}\Delta_{\varphi}R\Big)\;\oo_{\varphi}^n.\label{eq:A2c}
\eeqn
Note that
\beqn
\dot{w}&=&\frac{\partial}{\partial t}(e^{\frac 12 F}|\nabla F|_{\varphi}^2)\no\\&=&\frac 12 \dot{F}(w-1)+2e^{\frac 12 F}\text{Re}(\nabla\Delta_\varphi R\cdot_\varphi\nabla F)-e^{\frac 12 F}\nabla^2R(\nabla F,\nabla F), \label{eq:12}
\eeqn
Therefore, (\ref{eq:12}) and (\ref{eq:A2c}) imply that
\beqn
\|z\|_{\kappa,t}^{\kappa}-\|z\|_{\kappa,0}^\kappa&=&\int_{0}^{t}\;dt\int_{M}\;\Big(\kappa qw^{\kappa q-1}\Big(\frac 12 \dot{F}(w-1)+2e^{\frac 12 F}\text{Re}(\nabla\Delta_{\varphi}R\cdot_{\varphi}\nabla F)\no\\&\quad&-e^{\frac 12 F}\nabla^2R(\nabla F,\nabla F)\Big)+w^{\kappa q}\Delta_{\varphi}R\;\Big)\oo_{\varphi}^n\no\\&:=&I_0+I_1+I_2+I_3+I_4.
\eeqn
We will estimate each term $I_i$. By direct calculation, we have
\beqn
I_0&=&\int_{0}^{t}\;dt\int_{M}\;\frac {\kappa q}{2}z^\kappa\dot{F}\;\oo_{\varphi}^n\leq C q\kappa\|z\|_{b\kappa}^{\kappa},\\
I_1&=&-\int_{0}^{t}\;dt\int_M \;\frac{q\kappa}{2} w^{\kappa q-1}\dot{F}\;\oo_\varphi^n\leq
 C q\kappa\|z\|_{b\kappa}^{\kappa},\label{eq:3.52}
\eeqn
 where $C$ depends on $B_{R,p+1}$. Moreover, we have
\beqn
I_2&=&\int_{0}^{t}\;dt\int_{M}\;2\kappa qw^{\kappa q-1}e^{\frac F2}\text{Re}(\Na\Delta_{\varphi }R\cdot_{\varphi}\Na F)\,\oo_\varphi^n\no\\&=&-2q\kappa\int_{0}^{t}\;dt\int_{M}\;\Na(w^{\kappa q-1}e^{\frac F2}\Na F)\Delta_{\varphi} R\;\oo_{\varphi}^n\no\\
&=&-2q\kappa\int_{0}^{t}\;dt\int_{M}\;\Big((\kappa q-1)w^{\kappa q-2}\Na w\cdot_{\varphi}\Na F e^{\frac F2}\Delta_{\varphi }R\no\\
&&+\frac 12w^{\kappa q-1}e^{\frac 12F}|\Na F|_{\varphi}^2\Delta_{\varphi} R+w^{\kappa q-1}e^{\frac F2}\Delta_{\varphi}F\Delta_\varphi R\Big)\;\oo_{\varphi}^n.\no
\eeqn
 Therefore we have
\beqn
I_2&\leq& 2q\kappa\int_{0}^{T}\;dt\int_{M}\;\Big((\kappa q-1)w^{\kappa q-\frac32}|\Na w|_{\varphi}|\Delta_{\varphi}R|+\frac 12w^{\kappa q}|\Delta_{\varphi}R|\no\\&&+w^{\kappa q-1}e^{\frac F2}|\Delta_{\varphi}F||\Delta_{\varphi}R|\Big)\;\oo_{\varphi}^n\no\\
&\leq& 2q\kappa\int_{0}^{T}\;dt\int_{M}\frac{2q\kappa-2}{2q\kappa-1}|\Na w^{q\kappa-\frac1 2}|_{\varphi}|\Delta_{\varphi}R|\,\oo_{\varphi}^n+Cq\kappa||z||_{\kappa b}^{\kappa}+Cq\kappa||w^{\kappa q-1}||_{a}\no\\
&\leq &Cq\kappa\Big(\int_{0}^{T}\;dt\int_{M}|\;\Na w^{\kappa q-\frac 12}|_\varphi^2\;\oo_{\varphi}^n\Big)^{\frac  12}+Cq\kappa\|z\|_{\kappa b}^{\kappa}+Cq\kappa\|z\|_{\frac{a(\kappa q-1)}{q}}^{\frac{\kappa q-1}{q}}, \label{eq:1.1}
\eeqn
where $C$ depends on $A_{R,2p},B_{R,p+1},\|F\|_\infty,\|n+\Delta_g\varphi\|_{2p(n-1)}$ and we used  the fact that $\Delta_\varphi F\in L^{2p}(M\times[0,T),\oo_\varphi^n\wedge dt)$ in the second inequality. In fact, we have
\beqn
\Delta_\varphi F&=&-R+\tr_\varphi\Rc\leq -R+C(g)\sum_{i=1}^{n}\frac{1}{1+\varphi_{i\overline{i}}}\no\\
&\leq&-R+C(g)(n+\Delta\varphi)^{n-1}e^{-F}=-R+C(g)\td v^{n-1}e^{-F},\label{eq:3.55}
\eeqn
where $\td v=n+\Delta_g\varphi$.  By Lemma \ref{theo:main3} we have $\td v\in L^{s_0}(M\times[0,T),\oo_\varphi^n\wedge dt)$ for any $s_0>1$. Therefore, we have
$\Delta_\varphi F\in L^{2p}(M\times[0,T),\oo_\varphi^n\wedge dt)$.

Using (2.31) of Li-Zhang-Zheng \cite{[LZZ]}, for any $q>0$ we have
\beqn &&
\int_{0}^{T}dt\int_M|\Na(w^{q+\frac12})|_\varphi^2\,\oo_\varphi^n\leq C(\oo_g,\|F\|_\infty)(q+\frac12)^2\int_{0}^{T}dt\int_M\,\Big(\frac {q+1}qw^{2q}R^2\no\\&\quad&+\frac {q+1}qw^{2q}\td v^{2n-2}+\frac1q w^{2q+1}|R|+\frac 1qw^{2q+1}\td v^{n-1}\Big)\,\oo_\varphi^n. \label{eq:1.2}
\eeqn
Combining (\ref{eq:1.2}) with (\ref{eq:1.1}), we have
\beqn
I_2&\leq& Cq\kappa(q\kappa -\frac12)\Big(\int_{0}^{T}\;dt\int_{M}\;\Big(\frac{q\kappa }{q\kappa -1}w^{2\kappa q-2}R^2+\frac{q\kappa }{q\kappa -1}w^{2q\kappa -2}\td  v^{2n-2}+\frac{1}{q\kappa-1}w^{2q\kappa -1}|R|\no\\&\quad&+\frac{1}{q\kappa-1}\td v^{n-1}w^{2q\kappa -1}\Big)\;\oo_{\varphi}^n\Big)^{\frac12}
+Cq\kappa\|z\|_{\kappa b}^{\kappa}+Cq\kappa\|z\|_{\frac{a(\kappa q-1)}{q}}^{\frac{\kappa q-1}{q}}\no\\
&\leq& Cq\kappa(q\kappa -\frac12)\Big(\frac{2q\kappa}{q\kappa -1}\Big)^{\frac12}\|w^{2\kappa q-2}\|_d^{\frac12}+Cq\kappa(q\kappa -\frac12)(q\kappa-1)^{-\frac 12}\|w^{2q\kappa-1}\|_r^{\frac12}\no\\&\quad&+Cq\kappa\|z\|_{\kappa b}^\kappa+Cq\kappa\|z\|_{\frac{a(\kappa q-1)}{q}}^{\frac{\kappa q-1}{q}},\no
\eeqn
where $C$ depends on $n,p,\oo_g,Q_F,A_{R,2p},B_{R,p+1},\varphi(0)$ and $T$. Hence, we have
\beqn
I_2&\leq &Cq\kappa(q\kappa -\frac12)\Big(\frac{2q\kappa}{q\kappa -1}\Big)^{\frac12}\|z\|_{\frac{d(2q\kappa -2)}{q}}^{\frac{q\kappa -1}{q}}+Cq\kappa(q\kappa -\frac12)(q\kappa-1)^{-\frac12}\|z\|_{\frac{r(2q\kappa-1)}{q}}^{\frac{2q\kappa-1}{2q}}\no\\&\quad&+Cq\kappa\|z\|_{\kappa b}^\kappa+Cq\kappa\|z\|_{\frac{a(q\kappa -1)}{q}}^{\frac{q\kappa -1}{q}}.\label{eq:3.56}
\eeqn
Moreover, we have
\beqn
I_3&=&-q\kappa\int_{0}^{t}\;dt\int_{M}\;w^{q\kappa-1}e^{\frac F2}\nabla^2R(\Na F,\Na F)\;\oo_{\varphi}^n\no\\
&\leq& q\kappa\int_{0}^{T}\;dt\int_{M}\;w^{q\kappa-1}e^{\frac F2}|\nabla^2R|_{\varphi}|\Na F|_{\varphi}^2\;\oo_{\varphi}^n\no\\
&\leq&q\kappa\int_{0}^{T}\;dt\int_{M}\;w^{q\kappa}|\Na^2R|_{\varphi}\;\oo_{\varphi}^n\no\\
&\leq& q\kappa\Big(\int_{0}^{T}\;dt\int_{M}\;w^{2q\kappa}\;\oo_{\varphi}^n\Big)^{\frac 12}\Big(\int_{0}^{T}\;dt\int_{M}\;|\Na^2R|_{\varphi}^2\;\oo_{\varphi}^n\Big)^{\frac 12}\no\\
&=&C(B_{R,p+1})q\kappa\|z\|_{2\kappa}^\kappa,\label{eq:3.57}
\eeqn
and
\beqn
I_4=\int_{0}^{t}\;dt\int_{M}\;w^{q\kappa}\Delta_{\varphi}R\;\oo_{\varphi}^ndt\leq C(B_{R,p+1})\|z\|_{b\kappa}^{\kappa}.\label{eq:3.58}
\eeqn
Combining the inequalities (\ref{eq:3.52}), (\ref{eq:3.56}), (\ref{eq:3.57}) and (\ref{eq:3.58}), we have
\beqn
\|z\|_{\kappa,t}^{\kappa}-\|z\|_{\kappa,0}^\kappa&\leq& Cq\kappa\|z\|_{b\kappa}^\kappa+Cq\kappa(q\kappa -\frac12)\Big(\frac{2q\kappa}{q\kappa -1}\Big)^{\frac12}\|z\|_{\frac{d(2q\kappa -2)}{q}}^{\frac{q\kappa -1}{q}}\no\\&\quad&+Cq\kappa(q\kappa -\frac12)(q\kappa-1)^{-\frac12}\|z\|_{\frac{r(2q\kappa-1)}{q}}^{\frac{2q\kappa-1}{2q}}+Cq\kappa\|z\|_{\frac{a(q\kappa -1)}{q}}^{\frac{q\kappa -1}{q}}+Cq\kappa\|z\|_{2\kappa}^\kappa,\no
\eeqn
where $C$ depends on $n,p,\oo_g,Q_F,A_{R,2p},B_{R,p+1},\varphi(0)$ and $T$. The lemma is proved.

\end{proof}

Using Lemma \ref{lem:G} and the parabolic Sobolev inequality Lemma \ref{lem:interF} , we can show  Lemma \ref{lem:F}.
\begin{proof}[Proof of Lemma \ref{lem:F}]
    Let $w=e^{\frac 12 F}|\Na F|_{\varphi}^2+1$ as above. By the inequality (4.4)-(4.6) of Chen-Cheng \cite{[CC3]} or (2.27) of Li-Zhang-Zheng \cite{[LZZ]}, we have
    \beqn
    \Delta_{\varphi}w&\geq&2e^{\frac F2}\Na_{\varphi} F\cdot_{\varphi} \Na \Delta_{\varphi}F-C(g,  \|F\|_\infty)\td v^{n-1} \,w-\frac 12 Rw+\frac 12R.\label{eq:B2}
    \eeqn
    Multiplying both sides of (\ref{eq:B2}) by $w^{2q}$ and integrating by parts, for any $q> 0$  we have
    \beqn
    &&\int_{0}^{T}\;dt\int_{M}\;2qw^{2q-1}|\Na w|_{\varphi}^2\;\oo_{\varphi}^n= \int_{0}^{T}\;dt\int_{M}\;-w^{2q}\Delta_{\varphi}w\;\oo_{\varphi}^n\no\\
    &\leq&  \int_{0}^{T}\;dt\int_{M}\;-2e^{\frac F2}\Na_{\varphi}F\cdot_{\varphi}\Na\Delta_{\varphi}Fw^{2q}+C\td v^{n-1}w^{2q}+ |R| w^{2q+1}\;\oo_{\varphi}^n\no\\
    &\leq&\int_{0}^{T}\;dt\int_{M}\; \Big(q w^{2q-1}|\Na w|_{\varphi}^2+(4q+2)w^{2q}e^{\frac 12 F}(\Delta_{\varphi}F)^2
    +w^{2q+1}|\Delta_{\varphi}F|\Big)\,\oo_{\varphi}^n\no\\&\quad&+\int_{0}^{T}\;dt\int_{M}\;\Big(C\td v^{n-1}w^{2q}+ |R|w^{2q+1}\Big)\;\oo_{\varphi}^n,\label{eq:B3}
    \eeqn
where in the last equality we used the inequality (4.19) of Chen-Cheng \cite{[CC3]}.
Note that
\beq
|\Delta_{\varphi}F|\leq |R|+|\tr_{\varphi}\Rc|\leq |R|+C(g, \|F\|_\infty)\td v^{n-1}. \label{eq:B11}
\eeq
Combining (\ref{eq:B3}) with (\ref{eq:B11}), we have
\beqn &&\int_{0}^{T}\;dt\int_{M}\; q w^{2q-1}|\Na w|_{\varphi}^2\,\oo_{\varphi}^n
\leq C(\oo_g, \|F\|_\infty)\int_{0}^{T}\;dt\int_{M}\; \Big((q+1) w^{2q}R^2\no\\&&+q w^{2q}\td v^{2n-2}
+w^{2q+1}|R|+\td v^{n-1}w^{2q+1}\Big)\,\oo_{\varphi}^n.\label{eq:A2b}
\eeqn
Set $z=w^{q+\frac 12}$. By the Sobolev inequality Lemma \ref{lem:interF}, we have
\beqn
&&\int_{0}^{T}\;dt\int_{M}z^{\beta}\oo_{\varphi}^n\leq C(n,\oo_g,\gamma,\|F\|_\infty)\sup_{t\in[0,T)}\|z\|_{\kappa,t}^{\theta\beta}\int_{0}^{T}\;dt\int_{M}\;\Big(|\Na z|_{\varphi}^2+z^2\Big)\;\oo_{\varphi}^n\no\\
&&\leq C\sup_{t\in[0,T)}\|z\|_{\kappa,t}^{\theta\beta}(q+\frac12)^2\int_{0}^{T}\;dt\int_{M}\;\Big(\frac{q+1}{q}z^{\frac{4q}{2q+1}}R^2\no\\&&\quad+z^{\frac{4q}{2q+1}}\td v^{2n-2}+\frac 1qz^2|R|+\frac 1p \td v^{n-1}z^2\Big)\,\oo_{\varphi}^n\no\\
&&\leq C(q+\frac12)^2\sup_{t\in[0,T)}\|z\|_{\kappa,t}^{\theta\beta}\Big(\frac{2q+1}{q}\|z\|_{\frac{4qd}{2q+1}}^{\frac{4q}{2q+1}}+\frac 1q\|z\|_{2r}^2\Big),\label{eq:3.62}
\eeqn where $C$ depends on $n,\oo_g,\gamma,\|F\|_\infty,A_{R,2p}, B_{R,p+1}$ and $\varphi(0)$. By Lemma \ref{lem:G}, we have
\beqn
\|z\|_{\kappa,t}^{\kappa}-\|z\|_{\kappa,0}^\kappa&\leq& C(e+1)\|z\|_{b\kappa}^\kappa+C(e+1) (e+\frac12)\Big(\frac{2e+2}{e}\Big)^{\frac12}\|z\|_{\frac{4de}{2q+1}}^{\frac{2e}{2q+1}}\no\\&\quad&+C(e+1)(e+\frac12)e^{-\frac12}\|z\|_{\frac{r(4e+2)}{2q+1}}^{\frac{2e+1}{2q+1}}+C(e+1)\|z\|_{\frac{2ae}{2q+1}}^{\frac{2e}{2q+1}}+C(e+1)\|z\|_{2\kappa}^\kappa,\no\\
\quad\label{eq:3.63}
\eeqn
where $\kappa>\frac{2}{2q+1}$ by Lemma \ref{lem:G} and $e=q\kappa+\frac12\kappa-1> 0$ . Combining (\ref{eq:3.62}) and (\ref{eq:3.63}), we get
\beqn
\|z\|_\beta&\leq& C^{\frac 1\beta}(q+\frac12)^{\frac 2\beta}\Big(\|z\|_{\kappa,0}^\kappa+C(e+1)\|z\|_{b\kappa}^\kappa+C(e+1) (e+\frac12)^2(\frac{2e+2}{e})^{\frac12}\|z\|_{\frac{4de}{2q+1}}^{\frac{2e}{2q+1}}\no\\&\quad&+C(e+1)(e+\frac12)^2e^{-\frac12}\|z\|_{\frac{r(4e+2)}{2q+1}}^{\frac{2e+1}{2q+1}}+C(e+1)\|z\|_{\frac{2ae}{2q+1}}^{\frac{2e}{2q+1}}+C(e+1)\|z\|_{2\kappa}^\kappa\Big)^{\frac\theta\kappa}\no\\&\quad&\cdot\Big(\frac{2q+1}{q}\|z\|_{\frac{4qd}{2q+1}}^{\frac{4q}{2q+1}}+\frac 1q\|z\|_{2r}^2\Big)^{\frac 1\beta}.\label{eq:B10}
\eeqn
In order to use the iteration argument, we need to choose the constants in (\ref{eq:B10}) satisfying
\beqs
\beta=2+\Big(1-\frac 2\gamma\Big)\kappa&>&c:=\max\Big\{b\kappa,\frac{4de}{2q+1},\frac{r(4e+2)}{2q+1}, \frac {2ae}{2q+1}, 2\kappa,\frac{4qd}{2q+1},2r\Big\}\no \\
\kappa&>&\frac{2}{2q+1},\no
\eeqs
or equivalently,
\beqn
&&\max\Big\{\frac{2r-2}{1-\frac2\gamma},\frac 2{2q+1},\frac{4qd-4q-2}{(2q+1)(1-\frac2\gamma)}\Big\}<\kappa<\no\\
&&\min\Big\{\frac{1}{2d-1+\frac2\gamma}(2+\frac{4d}{2q+1}),\frac{1}{2r-1+\frac2\gamma}(2+\frac{2r}{2q+1}),\frac{2}{1+\frac2\gamma}\Big\}.
\eeqn
By Lemma \ref{lem:B1} in the Appendix B, for any $q$  with $\frac{1+\frac{2d}{\ga}}{2(d-1)}\leq q\leq\frac{r}{2(d-r)}$, we have
\begin{enumerate}
\item[(1)] There exists $\ka>0$ satisfying
\beq
\beta>c.
\eeq
\item[(2)]
\beqn
&&\max\Big\{\frac{2r-2}{1-\frac2\gamma},\frac 2{2q+1},\frac{4qd-4q-2}{(2q+1)(1-\frac2\gamma)}\Big\}=\frac{2r-2}{1-\frac2\gamma},\no\\
&&\min\Big\{\frac{1}{2d-1+\frac2\gamma}(2+\frac{4d}{2q+1}),\frac{1}{2r-1+\frac2\gamma}(2+\frac{2r}{2q+1}),\frac{2}{1+\frac2\gamma}\Big\}\no\\
&&=\frac{1}{2d-1+\frac2\gamma}(2+\frac{4d}{2q+1}).
\eeqn
  \item[(3)]   If $\kappa\to \frac{2r-2}{1-\frac2\gamma}$ and $q\to\frac{r}{2(d-r)}$, we have $\beta(q+\frac12)\to 2p$.
\end{enumerate}

 Therefore (\ref{eq:B10}) implies that
 \beqn
 \|z\|_\beta\leq C\|z\|_{c},\label{eq:3.67}
 \eeqn
 where $C$ depends on $n,q,\oo_g,\kappa,\gamma,\|F\|_\infty,A_{R,2p},B_{R,p+1}$ and $\varphi(0)$. Taking the $(q+\frac12)$-root in (\ref{eq:3.67}), we get
 \beqn
 \|w\|_{\beta(q+\frac12)}\leq C\|w\|_{c(q+\frac12)}.\label{eq:3.68}
 \eeqn
  Therefore by Property (3) and  (\ref{eq:3.68}) we have that for any $s:=\beta(q+\frac12)< 2p$, there exists \beq k:=c(q+\frac12)<\beta(q+\frac12)=s\eeq such that
 \beqn
 \|w\|_s\leq C(n,\oo_g,s,A_{R,2p},B_{R,p+1},\gamma,\kappa,\|F\|_\infty,\varphi(0))\|w\|_k.\label{eq:1.3}
 \eeqn
 By the interpolation inequality, we have
 \beqn
\|w\|_k\leq C(\ee)\|w\|_1+\ee \|w\|_s.\label{eq:1.4}
\eeqn
Combining (\ref{eq:1.3}) with (\ref{eq:1.4}) and choosing $\ee$ small enough, we get
\beqn
\|w\|_s\leq C\|w\|_1.\label{eq:3.70}
\eeqn
Note that
\beqn
\|w\|_1&=&\int_{0}^{T}dt\int_{M}\;\Big(e^{\frac F2}|\Na F|_{\varphi}^2+1\Big)\;\oo_{\varphi}^n\no\\
&\leq& C(\|F\|_\infty)\int_{0}^{T}\;dt\int_{M}\;|\Na F|_{\varphi}^2\;\oo_{\varphi}^n+C(\|F\|_\infty)\vol_{\oo_g}(M)T\no\\
&=&-C\int_{0}^{T}\;dt\int_{M}\;F\;\Delta_\varphi F\;\oo_{\varphi}^n+C\vol_{\oo_g}(M)T\no\\
&\leq& C(\oo_g,\|F\|_\infty,A_{R,2p},\|n+\Delta_g\varphi\|_{2p(n-1)},T).\no
\eeqn
Since $n+\Delta_g \varphi\in L^{2p(n-1)}(M\times[0,T),\oo_\varphi^n\wedge dt)$ by Lemma \ref{theo:main3}, we finish this proof.

\end{proof}

\subsection{Estimates of $\|\nabla\varphi\|_\infty$}
In this subsection, we show that $\|\nabla\varphi\|_\infty$ is bounded. First, we
recall the following result from Chen-Cheng \cite{[CC3]}, see also Lemma 2.5 in Li-Zhang-Zheng \cite{[LZZ]}.

    \begin{lem}\label{lem:equ}(cf. \cite{[CC3]}, \cite[Lemma 2.5]{[LZZ]})
        Let
        \beqs
        A(F,\varphi)&=&-(F+\la\varphi)+\frac{1}{2}\varphi^2,\\
        u&=&e^A(|\nabla\varphi|_g^2+10),
        \eeqs where $\la$ depends only on $\|\varphi\|_\infty$ and $\oo_g$.
        Then we have
        the inequality
        \beqs
        \Delta_{\varphi}u\geq \hat R u+\frac 1{n-1}|\Na \varphi|_g^{2+\frac 2n}e^{-\frac Fn}e^A,
        \eeqs where $\hat R=R-\lambda n(n+2) +(n+2)\varphi$.

    \end{lem}

    Using the equation (\ref{eq000}) of Calabi flow, we have the result.
    \begin{lem}\label{lem:3.14}
        Let $z=u^q(q>1)$ where $u$ is defined in Lemma \ref{lem:equ}. We have
        \beqn
         \|z\|_{\kappa,t}^{\kappa}- \|z\|_{\kappa, 0}^{\kappa}\leq C\|z\|_{2\kappa}^\kappa +C\|z\|_{b\kappa}^\kappa,
        \eeqn
        where C depends on $n,\oo_g,Q_F,A_{R,2p}$, $B_{R,p+1}$, $\|\varphi\|_\infty, \|F\|_\infty,\varphi(0)$ and $T$. Here $b$ and $p$ satisfy the equality $\frac 1{p+1}+\frac 1b=1.$
    \end{lem}
    \begin{proof}
        Taking the derivative with respect to $t$, we have
        \beqn
        \frac{\partial}{\partial t}\|z\|_{\kappa,t}^{\kappa}=\frac{\partial}{\partial t}\int_M\;|z|^{\kappa}\;\oo_\varphi^n=\int_M\;\Big(\kappa z^{\kappa-1}\dot{z}+z^{\kappa}\dot{F}\Big)\;\oo_\varphi^n.\no
        \eeqn
            Using $\dot{z}=qu^{ q-1}\dot{u}$ and
        \beqn &&\dot{u}=\dot{A}u+2e^A\text{Re}(\nabla R\cdot\nabla\varphi ),\no\\ &&\dot{A}=-(\dot{F}+\lambda\dot{\varphi})+\varphi\dot{\varphi},\no
        \eeqn
        where $\Na R\cdot \Na\varphi$ is taken with respect to $\oo_g$, we have
        \beqn
        \|z\|_{\kappa,t}^{\kappa}- \|z\|_{\kappa, 0}^{\kappa}&\leq&\int_{0}^{T}\;dt\int_M\Big(\;\kappa z^{\kappa-1}qz^{\frac{q-1}{q}}\Big(\dot{A}u+2e^A\text{Re}(\nabla R\cdot\nabla\varphi)\Big)+\dot{F}z^{\kappa}\Big)\;\oo_{\varphi}^n\no\\
        &:=&J_1+J_2+J_3.\no
        \eeqn
        We  estimate $J_1,J_2 $ and $J_3$ respectively. Note that
        \beqn
        J_1=\int_{0}^{T}\;dt\int_{M}\;\kappa q z^{\kappa}\dot{A}\;\oo_{\varphi}^n&\leq& C(\|\varphi\|_\infty,A_{R,2p},B_{R,p+1})q\kappa \Big(\int_{0}^{T}\;dt\int_{M}\;z^{b\kappa}\;\oo_{\varphi}^n\Big)^{\frac1b}\no\\&=&Cq\kappa \|z\|_{b\kappa}^{\kappa},\label{eq:3.79}
        \eeqn where $b$ and $p$ satisfy $\frac 1{p+1}+\frac 1b=1,$
        and
       \beqn
        J_2&=&2q\kappa \int_{0}^{T}\;dt\int_{M}\;z^{\kappa-\frac 1q}e^A\text{Re}(\Na R\cdot\nabla\varphi)\;\oo_{\varphi}^n\no\\
        &\leq& 2q\kappa \Big(\int_{0}^{T}\;dt\int_{M}\;z^{2\kappa-\frac 2q}e^{2A}|\nabla\varphi|_g^2\;\oo_{\varphi}^n\Big)^{\frac 12}\Big(\int_{0}^{T}\;dt\int_{M}\;|\Na R|_g^2\;\oo_{\varphi}^n\Big)^{\frac 12}.\label{eq:2.6}
\eeqn
Note that
   \beqn
        \int_{0}^{T}\;dt\int_M|\Na R|_g^2\,\oo_\varphi^n&\leq& C(\|F\|_\infty)\int_{0}^{T}\;dt\int_M |R\Delta_g R|\,\oo_g^n\no\\
        &\leq&C\int_{0}^{T}\;dt\int_M|R||\Na^2R|_\varphi(n+\Delta_g\varphi)\,\oo_g^n\no\\
        &\leq&C\Big(\int_{0}^{T}\;dt\int_M|R|^{2p}\,\oo_\varphi^n\Big)^{\frac{1}{2p}}
\Big(\int_{0}^{T}\;dt\int_M|\Na^2R|_\varphi^2\,\oo_\varphi^n\Big)^{\frac12}\no\\&&
\Big(\int_{0}^{T}\;dt\int_M\td v^{s}\,\oo_\varphi^n\Big)^{\frac1s},\label{eq:2.8}
        \eeqn where $s$ and $p$ satisfy $\frac 1{2p}+\frac12+\frac 1s=1$.
Combining (\ref{eq:2.6}) with (\ref{eq:2.8}), we have
        \beqn
        J_2        &\leq& C(\|F\|_\infty,A_{R,2p},B_{R,p+1},\|n+\Delta_g\varphi\|_s)q\kappa \Big(\int_{0}^{T}\;dt\int_{M}\;z^{2\kappa-\frac 2q}u\;\oo_\varphi^n\Big)^{\frac 12}\no\\&=&Cq\kappa \Big(\int_{0}^{T}\;dt\int_{M}\;z^{2\kappa-\frac 1q}\;\oo_{\varphi}^n\Big)^{\frac 12}\no\\
        &\leq& C(\|\varphi\|_\infty,\|F\|_\infty,A_{R,2p},B_{R,p+1},\|n+\Delta_g\varphi\|_s)q\kappa \Big(\int_{0}^{T}\;dt\int_{M}\;z^{2\kappa}\;\oo_{\varphi}^n\Big)^{\frac 12}\no\\&=&Cq\kappa \|z\|_{2\kappa}^{\kappa}.\label{eq:3.80}
        \eeqn
        Moreover, we have
        \beqn
        J_3=\int_0^T\,dt\int_M\,\dot{F}z^{\kappa}\,\oo_{\varphi}^n\leq C(B_{R,p+1})\Big(\int_{0}^{T}\;dt\int_{M}\;z^{b\kappa}\;\oo_{\varphi}^n\Big)^{\frac 1b}=C\|z\|_{\kappa b}^{\kappa}.\label{eq:3.81}
        \eeqn
        Combining (\ref{eq:3.79}), (\ref{eq:3.80}) with (\ref{eq:3.81}), we get
        \beqn
            \|z\|_{\kappa,t}^{\kappa}- \|z\|_{\kappa, 0}^{\kappa}\leq Cq\kappa\|z\|_{b\kappa}^\kappa+Cq\kappa\|z\|_{2\kappa}^\kappa,\no
        \eeqn
        where $C$ depends on $n,\oo_g,Q_F,A_{R,2p}$, $B_{R,p+1}$, $\|\varphi\|_\infty, \|F\|_\infty,\varphi(0)$ and $T$. The lemma is proved.
    \end{proof}

    Using Lemma \ref{lem:3.14} and Lemma \ref{lem:interF}, we have the result.
    \begin{lem}\label{lem:3.12}Under the assumption of Lemma \ref{theo:main3}, we have
        \beq
        |\Na \varphi(x, t)|_g\leq C, \label{eq:1.6}
        \eeq
where C depends on $n,\oo_g,\|\varphi\|_\infty,\|F\|_\infty,Q_F,A_{R,2p},B_{R,p+1},\varphi(0)$ and $T$.
    \end{lem}
    \begin{proof}Let $q> 1$. Since by Lemma \ref{lem:equ} $u=e^A(|\nabla\varphi|_g^2+10)$ satisfies
        \beqn
        \Delta_{\varphi}u\geq \hat R u+h, \label{eq:u}\no
        \eeqn where $h=\frac 1{n-1}|\Na \varphi|_g^{2+\frac 2n}e^{-\frac Fn}e^A,$ multiplying both
        sides by $u^{q-1}$ and integrating by parts we have
        \beqn
        &&\frac {4(q-1)}{q^2}\int_{0}^{T}dt\int_M\; |\Na (u^{\frac {q}2})|_\varphi^2\,\oo_{\varphi}^n =
        (q-1)\int_{0}^{T}dt\int_M\; u^{q-2}|\Na u|_\varphi^2\,\oo_{\varphi}^n\no\\
        &&=-\int_{0}^{T}dt\int_M\;u^{q-1}\Delta_{\varphi}u\,\oo_{\varphi}^n
        \leq-\int_{0}^{T}dt\int_M\;\Big(\hat R u^{q}+hu^{q-1}\Big)\,\oo_{\varphi}^n\no\\
        &&\leq\int_{0}^{T}dt\int_M\;|\hat R| u^{q}\,\oo_{\varphi}^n.\no
        \eeqn
        Letting $z=u^{\frac q2}$and using the Sobolev inequality Lemma \ref{lem:interF}, we have
        \beqn
        \int_{0}^{T}\;dt\int_M\;|z|^{\beta}\;\oo_\varphi^n&\leq& C(n,\oo_g,\gamma,\|F\|_\infty)\sup_{t\in[0,T)}\;\|z\|_{\kappa,t}^{(1-\frac{2}{\gamma})
\kappa}   \int_{0}^{T}dt\int_M\;\Big(|\nabla z|_\varphi^2+z^2\Big)\;\oo_\varphi^n\no\\
        &\leq& C\sup_{t\in[0,T)}\;\|z\|_{\kappa,t}^{(1-\frac{2}{\gamma})\kappa}\int_{0}^{T}\;dt\int_M\;(|\hat R|+1) u^{q}\,\oo_{\varphi}^n.\no
        \eeqn
        by Lemma \ref{lem:3.14}, we have
        \beqn
        \|z\|_{\kappa,t}^\kappa\leq \|z\|_{\kappa,0}^\kappa+Cq\kappa\|z\|_{b\kappa}^\kappa+Cq\kappa\|z\|_{2\kappa}^\kappa.\no
        \eeqn
        Therefore, we have
        \beqn
        \|z\|_{\beta}&\leq& Cq^{\frac\theta\kappa}\kappa^{\frac\theta\kappa}\Big(\|z\|_{\kappa,0}^{\kappa}+\|z\|_{b\kappa}^{\kappa}+\|z\|_{2\kappa}^{\kappa}\Big)^{\frac \theta \kappa}\|z\|_{2r}^{\frac{2}{\beta}}\no\\
        &\leq& Cq^{\frac{\theta}{\kappa}}\kappa^{\frac{\theta}{\kappa}}\Big(\sup_{x\in M}\Big(e^A(|\nabla\varphi|_g^2(x,0)+10)\Big)\vol_{\oo_g}(M)+2q\kappa\|z\|_{2\kappa}^{\kappa}\Big)^{\frac\theta\kappa}\|z\|_{2r}^{\frac 2\beta}\no\\
        &\leq& Cq^{\frac{\theta}{\kappa}}\|z\|_{2\kappa}^{\theta}\|z\|_{2r}^{\frac 2\beta},\no
        \eeqn
        where $C$ only depends on $n,\kappa,\gamma,\oo_g,\|F\|_\infty,\|\varphi\|_\infty,A_{R,2p},B_{R,p+1},\varphi(0)$ and $T$.
       By (\ref{eq:3.41}) we have
  $
        \beta>\max\{2\kappa,2r\}.
       $
    We conclude that if $q$ is large enough, then
        \beqn
        \|z\|_{\beta}\leq Cq^{\frac{\theta}{\kappa}}\|z\|_{2\max\{\kappa,r\}},\no
        \eeqn
        or equivalently,
        \beqn
        \|u\|_{\frac{q\beta}{2}}\leq C^{\frac{2}{ q}}q^{\frac{2\theta}{q\kappa}}\|u\|_{q\max\{\kappa,r\}}.\label{eq:3.99}
        \eeqn
        Letting $\theta_1=\frac \beta{\max\{2\kappa,2r\}}>1$ and $q_n=\frac{2}{\max\{r,\kappa\}}\theta_1^n$, the inequality (\ref{eq:3.99}) implies that
        \beqn
        \|u\|_{q_{n+1}\max\{r,\kappa\}}\leq C^{\frac 2{q_n}}q_n^{\frac{2\theta}{ q_n \kappa}}\|u\|_{q_n\max\{r,\kappa\}}.\no
        \eeqn
        Since $\kappa<\frac2{1+\frac 2\gamma}<\frac {2n}{2n-1}<2$ and $q_0=\frac{2}{\max\{r,\kappa\}}>1$, the standard Moser iteration argument shows that
        \beq
        \|u\|_\infty\leq C\|u\|_2\label{eq:3.100}
        \eeq
for some constant $C$ depending on $n,\kappa,\gamma,\oo_g,\|F\|_\infty,\|\varphi\|_\infty,A_{R,2p},B_{R,p+1},\varphi(0)$ and $T$. By the interpolation inequality Lemma \ref{lem:2.1}, we have
\beqn
\|u\|_2\leq \|u\|_1^{\frac 12}\|u\|_\infty^\frac 12.\label{eq:3.101}
\eeqn
Combining (\ref{eq:3.100}) and (\ref{eq:3.101}), we get
\beqn
\|u\|_\infty \leq C\|u\|_1.\label{eq:1.5}
\eeqn
Next we show that $\|u\|_1$ is bounded.
\beqn
\|u\|_1&=&\int_{0}^{T}dt\int_Me^A(|\Na\varphi|_g^2+10)\,\oo_\varphi^n\no\\
&\leq& C(\|\varphi\|_\infty,\|F\|_\infty)\int_{0}^{T}dt\int_M(|\Na\varphi|_g^2+10)\,\oo_g^n\no\\
&\leq& 10CT\cdot \vol_{\oo_g}(M)+C\int_{0}^{T}dt\int_M|\varphi\Delta_g\varphi|\,\oo_g^n.\label{eq:3.103}
\eeqn
Since $|\Delta_g\varphi|\leq |\Na^2\varphi|_\varphi (n+\Delta_g\varphi)$, we have
\beqn
\int_{0}^{T}dt\int_M|\varphi\Delta_g\varphi|\,\oo_g^n&\leq& C(\|\varphi\|_\infty)\Big(\int_{0}^{T}dt\int_M|\Na^2\varphi|_\varphi^2\,\oo_\varphi^n\Big)^\frac12\Big(\int_{0}^{T}dt\int_M(n+\Delta_g\varphi)^2\,\oo_\varphi^n\Big)^\frac12\no\\
&=&C\Big(\int_{0}^{T}dt\int_M|\Delta_\varphi \varphi|^2\,\oo_\varphi^n\Big)^\frac12\Big(\int_{0}^{T}dt\int_M(n+\Delta_g\varphi)^2\,\oo_\varphi^n\Big)^\frac 12.\no\\
\quad\label{eq:3.94}
\eeqn
Since $\Delta_\varphi \varphi=n-\tr_\varphi\oo_g\leq n+\td v^{n-1}e^{-F} $ by (\ref{eq:3.55}), we conclude that the right-hand side of (\ref{eq:3.94}) is bounded. Therefore, (\ref{eq:3.103}) implies that $\|u\|_1$ is bounded and by (\ref{eq:1.5}) we have (\ref{eq:1.6}). The lemma is proved.

    \end{proof}
     \subsection{Estimates of $\|n+\Delta_g\varphi\|_\infty$}
     In this section, we show the   estimate of $\Delta_{g}\varphi$. First, we
recall the following result from Chen-Cheng \cite{[CC3]}, see also Lemma 2.8 in Li-Zhang-Zheng \cite{[LZZ]}.
    \begin{lem}\label{lem:v}(cf. \cite{[CC3]}, \cite[Lemma 2.8]{[LZZ]}) Let
        \beq
        v=e^{-\alpha(F+\lambda\varphi)}(n+\Delta_g\varphi).\no
        \eeq
        Let $q>1$ and $\al>1$. There exists a constant $C(\oo_g)$ such that for $\la>C(\oo_g)$ we have
        \beqn &&
        \frac {3(q-1)}{q^2}\int_M\;  |\Na v^{\frac q2}|_{\varphi}^2\;\oo_{\varphi}^n\leq \int_M\; \Big( \td f+\frac {\al\la}{\al-1}+\frac 1n e^{-\frac Fn}R_g \Big)v^{q}\oo_{\varphi}^n\no\\
        &&+2q\int_M\;   v^{q}|\Na F|_\varphi^2\,\oo_\varphi^n+\frac {2\al^2\la^2q}{(\al-1)^2}\int_M\; e^B v^{q-1}|\Na \varphi|_g^2\,\oo_g^n,\label{eq:B01}
        \eeqn where $B=(1-\al)F-\al \la \varphi$ and $\td f=\alpha (\lambda n-R).$

    \end{lem}

    Combining Lemma \ref{lem:v}, Lemma \ref{lem:3.7} with Lemma \ref{lem:interF}, we have the result.

    \begin{lem}\label{lem:key} If $A_{R,2p}$, $B_{R,p+1}$ are bounded for some $p>n$, and $Q_F$ is  bounded, then there exists a constant C depending on $n,\oo_g,Q_F,A_{R,2p},B_{R,p+1},\|\varphi\|_\infty,\|F\|_\infty,\varphi(0)$ and $T$ such that\beqn
        n+\Delta_g\varphi\leq C.
        \eeqn
    \end{lem}
    \begin{proof}Since $n+\Delta_g \varphi\geq n e^{\frac Fn}$, we have
        \beq
        v^{q-1}=e^{\alpha(F+\lambda\varphi)}\frac {v^q}{n+\Delta_g\varphi}\leq \frac 1n e^{\alpha(F+\lambda\varphi)-\frac Fn}\,v^q.\label{eq:B08}
        \eeq
        Taking $z=v^{\frac q2}$ and $\al=2$ in  the inequality (\ref{eq:B01}), we have
        \beqn  &&
        \frac {3(q-1)}{q^2}\int_{0}^{T}\;dt\int_M\;  |\Na z|_{\varphi}^2\;\oo_{\varphi}^n\leq \int_{0}^{T}\;dt\int_M\; \Big(\td f+2\la+\frac 1n e^{-\frac Fn}R_g\Big)z^2\oo_{\varphi}^n\no\\
        &&+2q\int_{0}^{T}\;dt\int_M\;   z^{2}|\Na F|_\varphi^2\,\oo_\varphi^n+8\la^2q\int_{0}^{T}\;dt\int_M\; e^B v^{q-1}|\Na \varphi|_g^2\,\oo_g^n\no\\
        &\leq&\int_{0}^{T}\;dt\int_M\; \Big(\td f+2\la+\frac 1n e^{-\frac Fn}R_g\Big)z^2\oo_{\varphi}^n\no\\
        &&+2q\int_{0}^{T}\;dt\int_M\;   z^{2}|\Na F|_\varphi^2\,\oo_\varphi^n+C(n, \oo_g, \|F\|_\infty, \|\varphi\|_\infty)q\int_{0}^{T}\;dt\int_M\;  v^q\,\oo_\varphi^n,\no
        \eeqn where we used (\ref{eq:B08}) and Lemma \ref{lem:3.12} in the last inequality. Thus, we have
        \beq
        \frac {3(q-1)}{q^2}\int_{0}^{T}\;dt\int_M\;  |\Na z|_{\varphi}^2\;\oo_{\varphi}^n\leq q\int_{0}^{T}\;dt\int_M\; Gz^2\,\oo_{\varphi}^n+
        2q\int_{0}^{T}\;dt\int_M\;   z^{2}|\Na F|_\varphi^2\,\oo_\varphi^n,\no
        \eeq where
        \beq
        G=\td f+2\la+\frac 1n e^{-\frac Fn}R_g+C(g, \|F\|_\infty, \|\varphi\|_\infty).\no
        \eeq
        By Lemma \ref{lem:interF}, we have
        \beqn
        \int_{0}^{T}\;dt\int_M\;|z|^\beta\;\oo_{\varphi}^n&\leq& C(n,\oo_g,\|F\|_\infty,\gamma)q^2\sup_{[0,T)}\|z\|_{\kappa,t}^{(1-\frac{2}{\gamma})\kappa}\int_{0}^{T}\;dt\int_{M}\;\Big(G+|\nabla F|_\varphi^2\Big)z^2\;\oo_{\varphi}^n.\no\\
        \quad \label{eq:2.1}
        \eeqn
        By Lemma \ref{lem:3.7}, we have
    \beqn
        \|z\|_{\kappa,t}^{\kappa}-\|z\|_{\kappa,0}^{\kappa}\leq Cq\kappa\Big(\|z\|_{r\kappa}^{\kappa}+\|z\|_{\kappa b}^{\kappa}+\|z\|_{2\kappa}^{\kappa}\Big)\leq Cq\kappa\|z\|_{2\kappa}^{\kappa}.\label{eq:2.2}
        \eeqn
        According to Lemma \ref{lem:F}, $|\Na F|_\varphi^2\in L^{s}(M\times[0,T),\oo_\varphi^n\wedge dt)$ for $2n<s<2p$. Combining (\ref{eq:2.1}) with (\ref{eq:2.2}), we have
        \beqn
        \|z\|_{\beta}&\leq& C^{\frac 1\beta}q^{\frac 2\beta}\Big(\|z\|_{\kappa,0}^\kappa+Cq\|z\|_{2\kappa}^{\kappa}\Big)^{\frac {\theta}{{\kappa}}}\Big(\|z\|_{2r}^2+\|z\|_{2h}^{2}\Big)^{\frac 1\beta}\no\\
        &\leq& Cq^{\frac 2\beta}q^{\frac\theta\kappa}\|z\|_{2\kappa}^{\theta}\Big(\|z\|_{2r}^2+\|z\|_{2h}^{2}\Big)^{\frac 1\beta},\no
        \eeqn
        where $C$ depends on $n,\oo_g,\kappa,\gamma,A_{R,2p},B_{R,p+1},\|\varphi\|_\infty,\|F\|_\infty$ and $\varphi(0) $. Here, $h$ and $s$ satisfy the equality $\frac 1h+\frac 1s=1$. We need that
        \beqn
        \beta>\max\Big\{2\kappa,2r,2h\Big\},\no
        \eeqn
        or equivalently,
        \beq
\max\Big\{\frac{2r-2}{1-\frac2\gamma},\frac{2h-2}{1-\frac2\gamma}\Big\}<\kappa<\frac{2}{1+\frac 2\gamma}.\no
        \eeq
        Note that $s< 2p$,  we need the inequality
        \beqn
        \frac{2h-2}{1-\frac 2\gamma}<\frac{2}{1+\frac 2\gamma}.\label{eq:B09}
        \eeqn
        We can choose $\gamma $ close to $\frac{2n}{n-1}$ such that (\ref{eq:B09}) holds. Then we have
        \beqn
        \|v\|_{\frac{q\beta}{2}}\leq C^{\frac{2}q}q^{\frac 2q(\frac 2\beta+\frac\theta\kappa)}\|v\|_{q\max\{h,\kappa\}}.\label{eq:3.118}
        \eeqn
        Letting $\theta_2=\frac{\beta}{2\max\{h,\kappa\}}>1$ and taking  $q_n=\frac{2}{\max\{h,\kappa\}}\theta_2^n$, the inequality (\ref{eq:3.118}) implies that
        \beqn
        \|v\|_{q_{n+1}\max\{h,\kappa\}}\leq C^{\frac 2{q_n}}q_n^{\frac 2{q_n}(\frac 2\beta+\frac \theta\kappa)}\|v\|_{q_n\max\{h,\kappa\}}.\no
        \eeqn
          Since $h<\frac{2n}{2n-1}<2$ and $q_0=\frac 2{\max\{h,\kappa\}}>1$, the standard Moser iteration shows
         \beqn
         \|v\|_\infty\leq C\|v\|_{q_0\max\{h,\kappa\}}.\label{eq:3.120}=C\|v\|_2.\no
         \eeqn
           Since $\|v\|_2$ is bounded by Lemma \ref{theo:main3},  we know that $v$ is bounded and the lemma is proved.

    \end{proof}
    \section{Proof of   Theorem \ref{theo:1.1}}
    \begin{proof}[Proof of   Theorem \ref{theo:1.1}]
    Firstly we show that $Q_F$ is bounded along the Calabi flow. Without loss of generality, we may assume that $\varphi(0)\in\cH_0$. Then we have that $\varphi(t)\in\cH_0$ by (\ref{eq:3.10}). According to Lemma 4.4 of \cite{[CC2]}, we have
    \beqn
    |J_{-\Rc}(\varphi)|\leq C(n, \oo_g)d_1(0,\varphi).
    \eeqn
    Combining (4.1) with the proof of  Lemma \ref{Lem:3.4}, we conclude that $J_{-\Rc}(\varphi)$ is  uniformly bounded along Calabi flow. Since $\int_MF\,\oo_\varphi^n=\cK(\varphi)-J_{-\Rc}(\varphi)$, we know that $\int_MF\,\oo_\varphi^n$ is uniformly bounded under Calabi flow. Therefore, $Q_F$ is bounded.

    By the assumption, we have that $A_{R,2p}^n,B_{R,p+1}^n$ are bounded for $p> n$. Combining this with the boundedness of $Q_F$,
    we know that  $\|\varphi\|_\infty$ and $\|F\|_\infty$ are bounded by Theorem \ref{theo:main1}. Moreover, combining Lemma \ref{theo:main3}, Lemma \ref{lem:F},  Lemma \ref{lem:3.12} and Lemma \ref{lem:key} we conclude that $\|n+\Delta\varphi\|_\infty$ is bounded. Therefore, there exists a constant $C>0$ such that for any $t\in [0, T)$
    \beqn
    \frac 1C\oo_g\leq \oo_\varphi\leq C\oo_g. \label{eq:1.8}
    \eeqn
    Note that $F$ satisfies the parabolic equation
    \beq
    \pd Ft-\Delta_{\varphi}F=K, \quad K:=\Delta_{\varphi}R+R-\tr_{\varphi}Ric(\oo_g).
    \eeq
    By the assumption of Theorem \ref{theo:1.1}, the inequality (\ref{eq:B11}) and Lemma \ref{lem:key}, we have
    \beqs
    \int_0^T\,dt\int_M\;|K|^{p+1}\,\oo_g^n&\leq& C(p, \|F\|_{\infty})\int_0^T\,dt\int_M\;\Big(|\Delta_{\varphi}R|^{p+1}+|R|^{p+1}+\td v^{(n-1)(p+1)}\Big)\,\oo_{\varphi}^n\\
    &\leq&C,\quad p>n.
    \eeqs
    Since $\oo_{\varphi}$ satisfies (\ref{eq:1.8}), by the H\"older estimates of parabolic equations (cf. Theorem \ref{theo:app2} in the appendix), we know that $F\in C^{\al}(M\times [\frac 12T, T), \oo_g)(\al\in (0, 1))$.
   This together with (\ref{eq:1.8}) implies that $\varphi\in C^{2, \al'}(M\times [\frac 12T, T), \oo_g)$ for any $\al'\in (0, \al)$ (cf. Chen-Wang \cite{[ChenWang]},  Y. Wang \cite{[WangYu]}).
    Therefore, by He \cite{[He4]} the Calabi flow can be extended past time $T$. The theorem is proved.

    \end{proof}

\begin{appendices}
\section{The H\"older estimates for parabolic  equations}
In the appendix, we recall the H\"older estimates of parabolic equations. The readers are referred to Lieberman \cite[Section 13, Chapter VI]{[Lieb]}, Guerand \cite[Corollary 1.2]{[Gue]}, or Vasseur \cite[Theorem 18]{[Vas]} for details.

We use the notations in Guerand \cite{[Gue]}. Let $r>0$ and $x_0\in \RR^d$. We denote by $B_r(x_0)$ the ball of radius $r$ centered at $x_0$. For $(x_0, t_0)\in \RR^d\times \RR$ we define the parabolic cylinder $Q_r(x_0, t_0)=B_r(x_0)\times (t_0-r^2, t_0)$ and $Q_r=B_r(0)\times (-r^2, 0)$.

\begin{theo}\label{theo:app}Let $u: Q_2\ri \RR$ be a solution of
\beq
\pd ut=\Na_x\cdot (A\Na_x u) +B\cdot\Na_xu+g,
\eeq where $A(x, t), B(x, t)$ and $g(x, t)$ satisfy the following conditions:
\begin{enumerate}
  \item[(1).] $A(x, t)$ is a bounded measurable matrix and satisfies an ellipticity condition for two positive constants $\la, \La$,
\beq
0<\la I\leq A\leq \La I,
\eeq
  \item[(2).] $B(x, t)$ is bounded, measurable and $|B|\leq \La,$
  \item[(3).] $g(x, t)$ is bounded, measurable  and satisfies
  \beq
  \|g\|_{ L^q(Q_2)}\leq 1, \quad q>\max\Big\{2, \frac {d+2}{2}\Big\}.\label{eq:a1}
  \eeq

\end{enumerate}
 Then we have
 \beq
 \|u\|_{C^{\al}(Q_1)}\leq C(d, \la, \La)(\|u\|_{L^2(Q_2)}+1), \label{eq:a2}
 \eeq where $\al$ depends only on $d, \la$ and $\La.$

\end{theo}

We can easily remove the bound (\ref{eq:a1}). In fact, letting $\td g=K^{-1}g$ with $K:=\|g\|_{ L^q(Q_2)}$ and $\td u=K^{-1}u$,  by (\ref{eq:a2}) we have
\beq
 \|\td u\|_{C^{\al}(Q_1)}\leq C(d, \la, \La)(\|\td u\|_{L^2(Q_2)}+1).
 \eeq Therefore, we have
 \beq
 \|u\|_{C^{\al}(Q_1)}\leq C(d, \la, \La)(\| u\|_{L^2(Q_2)}+\|g\|_{ L^q(Q_2)}).
 \eeq

\begin{theo}\label{theo:app2}Let $(M, g)$ be a Riemannian manifold of dimension $d$ and $Q_r=B_r(x_0)\times (t_0-r^2, t_0)$, where $B_r(x_0)\subset M$ denotes the ball  centered at $x_0\in M$ of radius $r>0$ with respect to the metric $g$. If $u: Q_2\ri \RR$ be a solution of
\beq
\pd ut=\Delta_{h}u +f,
\eeq where $h(x, t)$ and $f(x, t)$ satisfy the following conditions:
\begin{enumerate}
  \item[(1).] $h(x, t)$ is a metric equivalent to $g$, i.e. there exist two constants $\la, \La>0$ such that
\beq
0<\la g\leq h\leq \La g,
\eeq

  \item[(2).] $f(x, t)$ is a bounded, measurable function  and satisfies $f\in L^q(Q_2)$ with $q>\max\{2, \frac {d+2}{2}\}$.

\end{enumerate}
 Then we have
 \beq
 \|u\|_{C^{\al}(Q_1)}\leq C(d, \la, \La, g)(\|u\|_{L^2(Q_2)}+\|f\|_{ L^q(Q_2)}), \label{eq:a2}
 \eeq where $\al$ depends only on $d, \la$ and $\La.$

\end{theo}
\begin{proof} We can choose a good coordinate chart with respect to the metric $g$, and the theorem follows from Theorem \ref{theo:app} by the standard argument. See, for example, Hebey \cite{[Heb]} or Metsch \cite{[Me]} for more details.

\end{proof}

\section{An elementary lemma}
In this section, we prove an elementary result which is used in Lemma \ref{lem:F}. In the proof of Lemma \ref{lem:F}, we have to choose suitable constants such that the iteration argument based on the inequality (\ref{eq:B10})  works. For readers' convenience, we summarize the restrictions on the constants as follows:
 \begin{enumerate}
   \item[(1).] The complex dimension $n$ of $M$ satisfies $n\geq 2$, and we assume $p>n;$
   \item[(2).]  $\te\in (0, 1)$ satisfies $(1-\te)\bb=2;$
       \item[(3).] $r$ satisfies $\frac 1{2p}+\frac 1r=1$;
       \item[(4).] $a$ satisfies $\frac 1a+\frac 1{2p}+\frac 1{p+1}=1$;

         \item[(5).] $b$ satisfies $\frac 1b+\frac 1{p+1}=1 $;

          \item[(6).] $d$ satisfies $\frac 1d+\frac 1p=1$;
          \item[(7).] $\ga\in \Big(2, \frac {2n}{n-1}\Big)$.  Since $p>n$, we choose $\ga$  close to $\frac {2n}{n-1}$ such that $\frac {2p}{p-1}<\ga<\frac {2n}{n-1}$. Therefore, we have
                            \beq \frac{r-1}{1-\frac 2\gamma}<\frac{1}{1+\frac 2\ga}. \label{eq:D4} \eeq
          \item[(8).] $\ka\in (0, 2)$ ;
          \item[(9).] We define $\bb=2+\Big(1-\frac 2{\ga}\Big)\ka$. We can check that $\bb$ satisfies $2<\bb<\ga.$

 \end{enumerate}
  We define
  \beqn
  A(p, q, \ga)&:=&\max\Big\{\frac{2r-2}{1-\frac2\ga},\frac{2}{2q+1},\frac{4qd-4q-2}{(2q+1)(1-\frac 2\ga)}\Big\},\\
  B(p, q, \ga)&:=&\min\Big\{\frac{1}{2d-1+\frac 2\ga}(2+\frac{4d}{2q+1}),\frac{1}{2r-1+\frac 2\ga}(2+\frac{2r}{2q+1}),\frac{2}{1+\frac 2\gamma}\Big\}.
  \eeqn With these notations, we have the result.
   \begin{lem}\label{lem:B1}
If $q$ satisfies
  \beq
  \frac{1+\frac{2d}{\ga}}{2(d-1)}\leq q\leq\frac{r}{2(d-r)},
  \eeq then we have
  \begin{enumerate}

    \item[(1).] $A(p, q, \ga)=\frac{2r-2}{1-\frac 2\ga}$;
    \item[(2).] $B(p, q, \ga)=\frac{1}{2d-1+\frac 2\ga}(2+\frac {4d}{2q+1}) $;
     \item[(3).] $0<A(p, q, \ga)<B(p, q, \ga)<2$, and we can choose $\kappa$ to be any number in the interval $(A(p, q, \ga), B(p, q, \ga))$;

\item[(4).] If $\kappa\to \frac{2r-2}{1-\frac 2\ga}$ and $q\to\frac{r}{2(d-r)}$, we have $\beta(q+\frac12)\to 2p$.
  \end{enumerate}

  \end{lem}

  \begin{proof}The proof is divided into several steps.

(1). We show that if $q$ satisfies  $\frac {2-\frac 2\ga-r}{2(r-1)}\leq q\leq\frac r{2(d-r)}$, then we have
    \beq
    A(p, q, \ga)=\frac{2r-2}{1-\frac 2\ga}.\label{eq:3.80a}
    \eeq
 In fact, if we solve the following equation for $q$
    \beqn
    \frac{2}{2q+1}=\frac{2r-2}{1-\frac 2\ga},\label{eq:D1}
    \eeqn
 then  we have $q=\frac{2-\frac 2\ga-r}{2(r-1)}$. Therefore, since the function $f(q):=\frac{2}{2q+1}$ is decreasing in $q$, we have that for $q>\frac{2-\frac 2\ga-r}{2(r-1)}$ the inequality holds
 \beq \frac 2{2q+1}<\frac {2r-2}{1-\frac 2\ga}. \label{eq:D2}\eeq  Similarly, if we solve the following equation for $q$
    \beq
    \frac{4qd-4q-2}{(2q+1)(1-\frac 2\ga)}=\frac {2r-2}{1-\frac 2\gamma},\no
    \eeq
    then we have $q=\frac {r}{2(d-r)}$. We can check that if $q$ satisfies $q<\frac {r}{2(d-r)}$, then
  \beq
  \frac{4qd-4q-2}{(2q+1)(1-\frac 2\ga)}<\frac {2r-2}{1-\frac 2\gamma}. \label{eq:D3}
  \eeq
  Note that
  \beq
  \frac {r}{2(d-r)}-\frac{2-\frac 2\ga-r}{2(r-1)}=\frac {2p-1}{\ga}>0,\no
    \eeq
  we have that (\ref{eq:3.80a}) follows from (\ref{eq:D2}) and  (\ref{eq:D3}).

 (2). We show that if $q\geq \frac{1+\frac{2d}{\ga}}{2(d-1)}$, then  we have
    \beq  B(p, q, \ga)=\frac{1}{2d-1+\frac 2\ga}\Big(2+\frac {4d}{2q+1}\Big).\label{eq:3.82}
    \eeq
  For simplicity, we define two functions of $q$:
  \beq
  f_1(q)=   \frac{1}{2d-1+\frac 2\ga}\Big(2+\frac{4d}{2q+1}\Big),\quad f_2(q)=\frac{1}{2r-1+\frac 2\ga}\Big(2+\frac{2r}{2q+1}\Big).
  \eeq By definition, we have
  \beq
  B(p, q, \ga)=\max\Big\{f_1(q), f_2(q), \frac{2}{1+\frac 2\gamma}\Big\}.
  \eeq
    Consider the following equation for $q$:
    \beq
    f_1(q)=\frac{2}{1+\frac 2\gamma}.
    \eeq
    Then we have $q=\frac{1+\frac{2d}{\ga}}{2(d-1)}$. Since $f_1(q)$ is  decreasing in $q$, we have $f_1(q)<\frac{2}{1+\frac 2\gamma}$ when $q>q_1:=\frac{1+\frac{2d}{\ga}}{2(d-1)}$. Similarly, we solve the following equation for $q$
    \beq
    f_2(q)=\frac{2}{1+\frac 2\gamma}.\no
    \eeq
    Then we have $q=\frac{2-r(1-\frac 2\ga)}{4(r-1)}$. Since $f_2(q)$ is  decreasing in $q$, we have $f_2(q) <\frac{2}{1+\frac 2\gamma}$ when $q$ satisfies $q>q_2:=\frac{2-r(1-\frac 2\ga)}{4(r-1)}$.  Since $r=\frac{2d}{d+1}$ by the assumption, we have $q_1=q_2$. Note that $f_1(q)$ and $f_2(q)$ are continuous functions on $(0, \infty)$, and the equation $f_1(q)=f_2(q)$ has only one solution $q=q_1$ on $(0, \infty)$. Therefore, $f_1(q)-f_2(q)$ is globally positive or globally negative on $(0, \infty).$ Note that
  \beq \lim_{q\to\infty}f_1(q)=
    \frac{2}{2d-1+\frac 2\ga}<\lim_{q\to\infty}f_2(q)=\frac{2}{2r-1+\frac 2\ga},\no\eeq
 We have that $f_2(q)>f_1(q)$ for all $q>q_1$. This implies that for $q>q_1$ we have $
  B(p, q, \ga)=f_1(q)
 $ and the equality (\ref{eq:3.82}) holds.

  (3).  By the assumption $n\geq 2$ and $\frac{2r-2}{1-\frac 2\gamma}<\frac{2}{1+\frac 2\ga}$, we have
  \beq
  \frac{2-\frac 2\ga-r}{2(r-1)}<\frac{1+\frac{2d}{\ga}}{2(d-1)}<\frac {r}{2(d-r)}.
  \eeq
    Therefore, if $q$ satisfies $\frac{1+\frac{2d}{\ga}}{2(d-1)}\leq q\leq\frac {r}{2(d-r)}$, then (1) and (2) imply that
    \beqn A(p, q, \ga)&=&\frac{2r-2}{1-\frac2\ga},\no\\ B(p, q, \ga)&=&\frac{1}{2d-1+\frac 2\ga}\Big(2+\frac {4d}{2q+1}\Big)\no.
    \eeqn
  In order to show $  A(p, q, \ga)<B(p, q, \ga) $, we calculate
  \beq\frac{1}{2d-1+\frac 2\ga}\Big(2+\frac{4d}{2q+1}\Big)\geq \frac{1}{2d-1+\frac 2\ga}\Big(2+\frac{4d}{\frac{d}{d-r}}\Big)=\frac{1}{2d-1+\frac 2\ga}(2+4d-4r).   \eeq
  Letting $D=1-\frac 2\ga$ and using $d=\frac r{2-r}$,  we have
  \beqn &&
  \frac{1}{2d-1+\frac 2\ga}(2+4d-4r)-\frac{2r-2}{1-\frac2\ga}\no\\
  &=&\frac {2+4d-4r}{2d-D}-\frac {2r-2}{D}\no\\
  &=&2\cdot \frac {(2d-r)D-2d(r-1)}{(2d-D)D}\no\\
  &=&\frac {2r^2}{D(2d-D)(2-r)}\Big(D-\frac {2(r-1)}{r}\Big).\label{eq:D6}
  \eeqn
  Note that
  \beqn
  2-r&=&\frac {2p-2}{2p-1}>0,\label{eq:D7}\\
  2d-D&=&\frac {2p}{p-1}-\Big(1-\frac 2\ga\Big)>0,\label{eq:D8}\\
  D-\frac {2(r-1)}{r}&=& 1-\frac 2\ga-\frac {2(r-1)}{r}=\frac 2r-\frac 2{\ga}-1>0,\label{eq:D9}
  \eeqn where we used the assumption (\ref{eq:D4}) in the last inequality. It follows from (\ref{eq:D7})-(\ref{eq:D9}) and (\ref{eq:D6}) that
   the right-hand side of (\ref{eq:D6}) is positive. This implies that  $A(p, q, \ga)<B(p, q, \ga) $ when  $q$ satisfies $\frac{1+\frac{2d}{\ga}}{2(d-1)}\leq q\leq\frac {r}{2(d-r)}$.

  (4). Since we assume $\ka\in (0, 2)$, we show that if $q$ satisfies $\frac{1+\frac{2d}{\ga}}{2(d-1)}\leq q\leq\frac {r}{2(d-r)}$ we have
\beq
\Big( A(p, q, \ga), B(p, q, \ga)\Big)\subset (0, 2). \label{eq:D5}
\eeq
It suffices to show that $B(p, q, \ga)< 2$. In fact,  $B(p,q,\ga)<2$ is equivalent to the inequality
\beqn
1+\frac{d}{2q+1}<d+\frac 1{\ga}.\label{eq:D10}
\eeqn
Note that
\beqn
d+\frac 1{\ga}-\frac{d}{2q+1}-1&=&\frac {(2d\ga+2-2\ga)q+1-\ga}{\ga(2q+1)}\no\\
&\geq& \frac 1{\ga(2q+1)}\Big((2d\ga+2-2\ga)\cdot \frac{1+\frac{2d}{\ga}}{2(d-1)}+1-\ga\Big)\no\\
&=& \frac d{\ga^2(d-1)(2q+1)}\Big((2d-1)\ga+2\Big)>0,\no
\eeqn where we used $d>1$ and $q\geq \frac{1+\frac{2d}{\ga}}{2(d-1)}$. Therefore, (\ref{eq:D10}) and also (\ref{eq:D5}) hold.
 Thus,  we can choose $\ka$ to be any number in $(A(p, q, \ga), B(p, q, \ga))$.

  (5). Note that if $\kappa\to \frac{2r-2}{1-\frac 2\ga}$ and $q\to\frac{r}{2(d-r)}$, we have \beqn
 \beta(q+\frac 12)=\Big(2+(1-\frac 2\ga)\kappa\Big)(q+\frac 12)&\to&\Big(2+(1-\frac 2\ga)\frac{2r-2}{1-\frac 2\ga}\Big)\frac {d}{2(d-r)}\no\\&=&\frac{dr}{d-r}=2p.\no
 \eeqn
The lemma is proved.

\end{proof}

\end{appendices}

    \vskip10pt
\noindent Haozhao Li, Institute of Geometry and Physics, and Key Laboratory of Wu Wen-Tsun
Mathematics, School of Mathematical Sciences, University of Science and Technology of China, No. 96 Jinzhai Road, Hefei, Anhui Province, 230026, China;  hzli@ustc.edu.cn.\\

\noindent    Linwei Zhang, School of Mathematical Sciences,
    University of Science and Technology of China, No. 96 Jinzhai Road, Hefei,
    Anhui Province, 230026, China;  zhanglinwei@mail.ustc.edu.cn.

\end{document}